\documentclass[reqno,12pt, oneside]{amsart}

\usepackage{enumerate}
\usepackage[nobysame]{amsrefs}
\usepackage{enumitem}
\usepackage{amssymb}
\usepackage{amsmath}
\usepackage{amsthm}
\usepackage{mathtools}
\usepackage{amscd}
\usepackage{microtype}
\usepackage[right=2cm,left=2cm, top=2.5cm, bottom=2.5cm,includehead]{geometry}
\usepackage{amsfonts}

\usepackage{tgpagella}
\usepackage{eulervm}

\usepackage[backgroundcolor=lightgray]{todonotes}

\theoremstyle{plain}
\newtheorem{theorem}{Theorem}
\newtheorem{corollary}[theorem]{Corollary}
\newtheorem{lemma}[theorem]{Lemma}
\newtheorem{proposition}[theorem]{Proposition}

\theoremstyle{remark}
\newtheorem{remark}[theorem]{Remark}
\newtheorem{example}[theorem]{Example}

\theoremstyle{definition}
\newtheorem{definition}[theorem]{Definition}
\newtheorem*{notation}{Notation}
\newtheorem*{question}{Open problem}
\newtheorem*{ackn}{Acknowledgements}

\DeclarePairedDelimiter{\abs}{\lvert}{\rvert}
\DeclarePairedDelimiter{\scal}{\langle}{\rangle}
\DeclarePairedDelimiter{\norm}{\lVert}{\rVert}

\DeclarePairedDelimiterX{\dset}[2]{\lbrace}{\rbrace}{#1\;\delimsize|\;#2}

\DeclareMathOperator{\conv}{conv}

\DeclareMathOperator{\var}{var}

\newcommand{\ball}{{\overline{B}}}

\newcommand{\R}{\mathbb R}

\newcommand{\N}{\mathbb N}
\newcommand{\f}{\varphi}

\DeclareMathOperator{\card}{card}
\DeclareMathOperator{\dist}{dist}

\title[Compactness in normed spaces: a unified approach through semi-norms]{Compactness in normed spaces:\\ a unified approach through semi-norms}

\author{Jacek Gulgowski}
\address[J.~Gulgowski]{Institute of Mathematics\\
Faculty of Mathematics, Physics and Informatics\\
University of Gda\'nsk\\
80-308 Gda\'nsk\\
Poland}
\email[J.~Gulgowski]{dzak@mat.ug.edu.pl}

\author{Piotr Kasprzak}
\address[P. Kasprzak]{Department of Nonlinear Analysis and Applied Topology\\
Faculty of Mathematics and Computer Science\\
  Adam Mickiewicz University in Pozna\'n\\
  ul.\ Uniwersytetu Pozna\'nskiego 4\\
  61-614 Pozna\'n\\
  Poland}
\email[P.~Kasprzak]{kasp@amu.edu.pl}

\author{Piotr Ma\'ckowiak}
\address[P. Ma\'ckowiak]{Department of Nonlinear Analysis and Applied Topology\\
Faculty of Mathematics and Computer Science\\
  Adam Mickiewicz University in Pozna\'n\\
  ul.\ Uniwersytetu Pozna\'nskiego 4\\
  61-614 Pozna\'n\\
  Poland}
\email[P.~Ma\'ckowiak]{piotr.mackowiak@amu.edu.pl}

\date{\today}

\begin{document}
\begin{abstract}
In this paper we prove two new abstract compactness criteria in normed spaces. To this end we first introduce the notion of an equinormed set using a suitable family of semi-norms on the given normed space satisfying some natural conditions. Those conditions, roughly speaking, state that the norm can be approximated (on the equinormed sets even uniformly) by the elements of this family. As we are given some freedom of choice of the underlying semi-normed structure that is used to define equinormed sets, our approach opens a new perspective for building compactness criteria in specific normed spaces.  As an example we show that natural selections of families of semi-norms in spaces $C(X,\R)$ and $l^p$ for $p\in[1,+\infty)$ lead to the well-known compactness criteria (including the Arzel\`a-Ascoli theorem). In the second part of the paper, applying the abstract theorems, we construct a simple compactness criterion in the space of functions of  bounded  Schramm variation.
\end{abstract}

\maketitle

\section{Introduction}
Compactness is one of the  most fundamental mathematical notions. Because of that, after more than a century from its formal introduction, it still attracts great interest of researchers. Compactness is so widespread that it seems nigh to impossible to even briefly mention all the theories where it plays a crucial role. Therefore, let us only draw the readers' attention to a few examples, which may be particularly interesting from an analyst point of view.

It is well-known that in infinite dimensional Banach spaces the Peano theorem on the existence of solutions of ordinary differential equations does not hold. It is even not so difficult to provide examples of differential equations with continuous right-hand sides which do not admit a local solution of the Cauchy problem. However, the situation changes completely (and local solutions to infinite-dimensional ODEs do exist), if besides the continuity assumption some `compactness conditions' on the right-hand side of the equation are added. Usually those conditions are expressed in terms of Kuratowski or Hausdorff measures of non-compactness and allow to prove the compactness of the set of approximate solutions to the given problem. (For more information on this topic see~\cite{BBK}*{Section~1.2}.) The same (general) phenomenon can be seen when we look at the Brouwer and Schauder (or some of its generalizations like Darbo or Sadovskii) fixed point theorems, where the latter result simply cannot hold without any compactness assumptions. (More details on compactness and fixed point theorems can be found in~\cite{BBK}*{Section~2.2}.) Of course one should not forget about the role of compactness in its various forms and shapes in functional analysis. Compact linear operators and their spectral properties, Montel spaces, that is, topological vector spaces with the Heine--Borel property (among which the spaces of holomorphic functions or smooth functions on open subsets are especially interesting), or compact embeddings of Sobolev-type spaces and Poincar\'e inequalities are just a few topics worth mentioning here.

Because of the importance of compactness in mathematics and the fact that very often it is not so easy or straightforward to prove that certain sets (or operators) are (or are not) compact, there is a great need for finding easy-to-use compactness criteria. Among the most famous are the Ascoli--Arz\`ela and Fr\'echet--Kolmogorov theorems in the spaces of continuous and $p$-Lebesgue integrable functions, respectively (see~\cite{precup}*{Section~1.2 and~1.3}). There are also less known relatives of those results for sequence spaces (cf., for example,~\cite{AKPRS}*{Section~1.1.9} or~\cite{BBK}*{Example~1.1.10}). Although all those criteria are formulated for different Banach spaces, they share some similarities. Namely, certain characteristics corresponding to the very nature of the considered space (connected with, for example, continuity for the space $C(X,\mathbb R)$, integral for $L^p(X,\mathbb R)$ or remainder of a series for $l^p$) must be small for all the elements of a given set so that this set and its description can be `reduced' to a family of finitely many points. This is of course understandable, as we work with compactness, and especially evident in the proofs of those results. 

Therefore, a natural question arises whether it is possible to give a unified approach to various known compactness criteria which in its formulation would use only general and basic functional analytic and no space-specific notions, but at the same time 
it would be quite general (in the sense that it would work in any normed/Banach space) and easy-to-use in applications. In the literature there are already quite general compactness criteria like, for example, the Hausdroff criterion (see e.g.~\cite{BBK}*{Theorem~1.1}) or a criterion for a separable Banach space $E$ which states that if $(E_n)_{n \in \mathbb N}$ is an increasing sequence of finite-dimensional subspaces of $E$ whose union is dense in $E$, then any bounded subset $A$ is relatively compact in $E$ if and only if $\lim_{n \to \infty} \sup_{a \in A}\dist(a,E_n)=0$ (cf. e.g.~\cite{schmidt}*{Proposition~1}). However, both of them have some imperfections. The Hausdorff criterion seems a little bit artificial, and therefore not easily applicable, as there is almost no difference between compactness and total boundedness. On the other hand, the second criterion works only in separable Banach spaces. There are some of its generalizations to weakly compactly generated spaces (cf.~~\cite{schmidt}*{Theorem~2}), but they become more technical and still they do not cover all the Banach spaces.        

The aim of this paper is to provide a new general compactness criterion satisfying all the above requirements. To this end we first introduce the notion of an equinormed set using a suitable family of semi-norms on the given normed space satisfying some natural conditions (see~Definition~\ref{def:equinormed} below). Then, we proceed with studying basic set-theoretic and topological properties of such sets. And finally, equipped with all the necessary facts, we prove our two main results providing the link between equinormed and relatively (precompact) sets (see Theorems~\ref{thm:compactness_ver1} and~\ref{thm:compactness_ver2}). It is worth mentioning here that our approach is motivated by the ideas and papers of Bugajewski and Gulgowski (see, e.g.~\cites{BG20, G21}), who proved new compactness criteria in various spaces of functions of bounded variation.   

As is shown by examples in Sections~4 and~5, from our abstract compactness criteria one can quite easily `recover' the compactness criteria in $C(X,\mathbb R)$ and $l^p$. However, our new results have (at least) one advantage over the classical results. Namely, they are much simpler from a methodological point of view. This is especially evident when we look at the compactness criterion in separable (or weakly compactly generated) Banach spaces mentioned above. Instead of calculating a limit (which involves an infinite sequence $(E_n)_{n \in \mathbb N}$ of finite-dimensional subspaces of $E$) or even a limit superior involving elements of the dual space $E^\ast$ in the more general version of that result, in the case of our results the verification whether the given set is precompact/relatively compact can be `reduced' to proving certain inequality. Moreover, we are also given some freedom of choice of the underlying semi-normed structure that is used to define equinormed sets.

In the last sixth section we show how to apply the abstract results to give a completely new compactness criterion in the space of functions of bounded Schramm variation. It is worth mentioning here that the Schramm variation generalizes the classical variations (Jordan, Wiener, Young) as well as the Waterman (also know as the $\Lambda$) variation. Because of that the results of the first part of Section~6 complete the studies conducted by Bugajewski and Gulgowski (see, e.g.~\cites{BG20, G21}) concerning compactness criteria in various spaces of functions of bounded variation. As a by-product of our research, some ambiguities concerning the definition of the Schramm variation are clarified and some errors appearing in the original paper of Schramm (see~\cite{schramm}) are fixed. (For a detailed discussion of problems and their solutions see Section~6.)


\section{Preliminaries}

The aim of this section is to introduce the notation and conventions used in this paper and to recall some basic definitions.

Throughout the paper we will be dealing with a normed space $(E, \norm{\cdot}_E)$ equipped not only with the default norm $\norm{\cdot}_E$ but also with various (semi-)norms $\norm{\cdot}_i$. Therefore, when talking for example about balls, we will always add some subscripts to indicate with respect to which of the (semi-)norms the ball in question is taken. Thus, the closed balls with center $x$ and radius $r>0$ in $(E, \norm{\cdot}_E)$ and $(E, \norm{\cdot}_i)$ will be denoted by $\ball_E(x,r)$  and $\ball_i(x,r)$, respectively. Similar convention will also apply to open balls $B_E(x,r)$ and $B_i(x,r)$. If $A$ is a subset of $E$, then by $\overline{A}$ and $\overline{A}^i$ we will denote the closure of $A$ with respect to the (semi-)norms $\norm{\cdot}_E$ and $\norm{\cdot}_i$, respectively. (Although $\norm{\cdot}_i$ in general will be a semi-norm, the closure of a set with respect to $\norm{\cdot}_i$ should be understood in the classical way, that is, as the set of all points $x \in E$ such that $\norm{x-x_n}_i \to 0$ for some sequence $(x_n)_{n \in \mathbb N}$ of elements of $A$.) 

Whenever we will say \emph{interval}, we will mean a non-empty and convex subset of $\mathbb R$. In other words, unless stated otherwise we also admit \emph{degenerate} closed intervals consisting of a single element. Having two intervals $I$ and $J$, we will say that they are \emph{non-overlapping} if their intersection $I \cap J$ consists of at most one point. For a function $x$ and an interval $I:=[a,b]$, to simplify the notation, we will write $x(I)$ for $x(b)-x(a)$. If $A\subseteq \mathbb R$, then by $\chi_A$ we will denote the characteristic function of the set $A$.

We will also need the notion of a Young function. Let us recall that a function $\varphi\colon [0,+\infty)\to[0,+\infty)$ is said to be a \emph{Young function} (or \emph{$\varphi$-function}) if it is convex and such that $\varphi(t)=0$ if and only if $t=0$. It is easy to check that every Young function is continuous and strictly increasing so that $\f(t)\to+\infty$ as $t\to+\infty$.

We will use default symbols to denote the classical spaces. And so, by $l^p$ we will denote the Banach space of all real $p$-summable sequences endowed with the norm $\norm{x}_{l^p}:=\bigl(\sum_{n=1}^\infty \abs{\xi_n}^p\bigr)^{1/p}$, where $p \in [1,+\infty)$ and $x=(\xi_n)_{n \in \mathbb N}$. 
Similarly, the Banach spaces of all real convergent and null sequences with the supremum norm $\norm{x}_{\infty}:=\sup_{n \in \mathbb N}\abs{\xi_n}$ will be denoted by $c$ and $c_0$, respectively; here as before, $x=(\xi_n)_{n \in \mathbb N}$. If $X$ is a compact metric space, then by $C(X,\mathbb R)$ we will denote the Banach space of all continuous real-valued functions defined on $X$ with the norm $\norm{x}_{\infty}:=\sup_{t \in X}\abs{x(t)}$. We use the same symbol to denote the norms in $c$, $c_0$ and $C(X,\mathbb R)$, but this should not lead to any confusion. We are also aware of the inconsistency in the way of denoting the norms in the spaces $l^p$ and $c_0$, $C(X,\mathbb R)$. However, when dealing with the sequence spaces, we will often encounter semi-norms indexed by natural numbers. Thus, using the symbol $\norm{\cdot}_p$ for the norm in $l^p$ could definitely lead to some mistakes. On the other hand, there is no point in not using the well-established symbol for the supremum norm.

Although the definitions of a relatively compact and precompact set are well-known, sometimes they are mistakenly used interchangeably. So, to avoid any ambiguity, we have decided to provide them. A metric space $X$ is called \emph{precompact} if its completion is compact. Equivalently, $X$ is precompact if and only if every sequence in $X$ contains a Cauchy subsequence, or if and only if $X$ is \emph{totally bounded}, that is, for each $\varepsilon>0$ there is a finite collection of points $a_1,\ldots,a_n \in X$ such that $X=\bigcup_{i=1}^n \ball_X(a_i,\varepsilon)$. Let us also recall that a subset of a metric space $X$ is precompact if it is precompact as a metric space itself (with the metric inherited from $X$). On the other hand, a subset $A$ of a metric space $X$ is called \emph{relatively compact} in $X$ if its closure in $X$ is compact, or, equivalently, if every sequence of elements in $A$ contains a subsequence convergent in $X$. In general, relative compactness is a stronger notion than precompactness, that is, if $A$ is a relatively compact subset of a metric space $X$, then it is also a precompact subset of $X$. However, in complete metric spaces those two notions coincide. (For more information on precompact and relatively compact sets see for example~\cite{MV}*{Chapter 4}.)

\section{Abstract compactness criterion}
\label{sec:abstract_compactness_criterion}

The main goal of this section is to prove a characterization of relatively compact and precompact subsets of a normed space whose norm can be ``approximated'' by a suitable family of semi-norms, that is, we will deal with a normed space $(E,\norm{\cdot}_E)$ and a family $\bigl\{\norm{\cdot}_i\bigr\}_{i \in I}$ of semi-norms on $E$ satisfying the following two conditions:
\begin{enumerate}[label=\textbf{\textup{(A\arabic*)}}]
 \item\label{i} $\norm{x}_E=\sup_{i \in I}\norm{x}_i$ for $x \in E$,
 \item\label{ii} for every $i,j \in I$ there is an index $k \in I$ such that $\norm{x}_i \leq \norm{x}_k$ and $\norm{x}_j \leq \norm{x}_k$ for $x \in E$; in other words, the family $\bigl\{\norm{\cdot}_i\bigr\}_{i \in I}$ forms a directed set. 
\end{enumerate}

\begin{remark}
Note that in any normed space the collection of the families of semi-norms satisfying conditions~\ref{i} and~\ref{ii} is non-empty. We can always take the family consisting of $\norm{\cdot}_E$ or its multiples, e.g. $(1-\frac{1}{n})\norm{\cdot}_E$ with $n\in \mathbb N$. Observe also that in a normed space $E$ the norm of an element $x$ can be expressed as $\norm{x}_E=\sup_{x^\ast \in \ball_{E^\ast}(0,1)}\abs{x^\ast(x)}$, where $E^\ast$ denotes the dual space of $E$, and therefore the family $\bigl\{\norm{\cdot}_A\bigr\}_{A \in I}$ of all the semi-norms of the form $\norm{\cdot}_A = \sup_{x^\ast \in A}\abs{x^\ast(\cdot)}$, where $A$ is a finite subset of $\ball_{E^\ast}(0,1)$, satisfies conditions~\ref{i} and~\ref{ii}.
\end{remark}

As it is often done in the context of locally convex spaces, we will now ``translate'' conditions~\ref{i} and~\ref{ii} expressed in terms of (semi-)norms into the language of geometry. We have decided to provide the proofs of the following two simple lemmas for the convenience of the reader.

\begin{lemma}\label{lem:geom_i}
Let $(E,\norm{\cdot}_E)$ be a normed space and let $\bigl\{\norm{\cdot}_i\bigr\}_{i \in I}$ be a family of semi-norms on $E$. The family $\bigl\{\norm{\cdot}_i\bigr\}_{i \in I}$ satisfies condition~\ref{i} if and only if\/ $\ball_E(0,1)=\bigcap_{i \in I}\ball_i(0,1)$.   
\end{lemma}

\begin{proof}
First, let us assume that the family $\bigl\{\norm{\cdot}_i\bigr\}_{i \in I}$ satisfies condition~\ref{i} and let $x\in \ball_E(0,1)$. Then, for any $i \in I$ we have $\norm{x}_i\leq \norm{x}_E \leq1$, which clearly implies that $x \in \bigcap_{i \in I} \ball_i(0,1)$. On the other hand, $\norm{x}_i \leq 1$ for $i \in I$ if $x\in\bigcap_{i \in I} \ball_i(0,1)$. Hence, $\norm{x}_E = \sup_{i \in I} \norm{x}_i \leq 1$.

Let us now assume that $\ball_E(0,1) =\bigcap_{i \in I} \ball_i(0,1)$. We are going to show that the family $\bigl\{\norm{\cdot}_i\bigr\}_{i \in I}$ satisfies~\ref{i}. Take any $x \in E\setminus\{0\}$. Because of the homogeneity, we may assume that $\norm{x}_E = 1$. Then, by our assumption, we have $\norm{x}_i \leq 1$  for every $i \in I$. If $r:=\sup_{i \in I} \norm{x}_i< 1$, putting $y = \frac{1}{r}x$, we would have $\norm{y}_E=\frac{1}{r}>1$ and $\norm{y}_i = \frac{1}{r}\norm{x}_i \leq 1$ for $i \in I$. This in turn would imply that $\ball_E(0,1) \neq \bigcap_{i \in I} \ball_i(0,1)$, which is impossible.
\end{proof}

\begin{remark}\label{rem:geom_i}
Note that the equivalence described in Lemma~\ref{lem:geom_i} is no longer true, if we replace closed balls with open ones; it is the ``right'' implication that fails. Indeed, consider the space $l^1$ with its standard norm $\norm{\cdot}_{l^1}$ and the family of semi-norms $\norm{x}_i:=\sum_{n=1}^i \abs{\xi_n}$, where $i \in \mathbb N$ and $x=(\xi_n)_{n \in \mathbb N} \in l^1$. Then, condition~\ref{i} is obviously satisfied and $y:=(2^{-n})_{n \in \mathbb N}$ belongs to $\bigcap_{i \in \mathbb N} B_i(0,1)$, but $y \notin B_{l^1}(0,1)$. 
\end{remark}

\begin{lemma}\label{lem:geom_ii}
Let $(E,\norm{\cdot}_E)$ be a normed space and let $\bigl\{\norm{\cdot}_i\bigr\}_{i \in I}$ be a family of semi-norms on $E$. The family $\bigl\{\norm{\cdot}_i\bigr\}_{i \in I}$ satisfies condition~\ref{ii} if and only if for every $i,j \in I$ there exists an index $k \in I$ such that $\ball_k(0,1) \subseteq \ball_i(0,1) \cap \ball_j(0,1)$.   
\end{lemma}

\begin{proof}
First, let us assume that the family $\bigl\{\norm{\cdot}_i\bigr\}_{i \in I}$ satisfies condition~\ref{ii}. Let $i,j \in I$ be arbitrary and choose $k$ according to~\ref{ii}. Then, $\norm{x}_i\leq 1$ and $\norm{x}_j\leq 1$ for any $x\in\ball_k(0,1)$. And consequently, $\ball_k(0,1) \subseteq \ball_i(0,1) \cap \ball_j(0,1)$.

Now, take arbitrary indices $i,j \in I$. By our assumption there is $k \in I$ with the property that $\ball_k(0,1) \subseteq \ball_i(0,1) \cap \ball_j(0,1)$. Suppose that there is a point $y \in E\setminus\{0\}$ such that $\norm{y}_j> \norm{y}_k$. Then, $\norm{y}_j > r>\norm{y}_k$ for some $r>0$. In such a case we have $\norm{\frac{1}{r}y}_k < 1$ and $\norm{\frac{1}{r}y}_j>1$, which means that $\frac{1}{r}y\in\ball_k(0,1)$ and $\frac{1}{r}y\not\in\ball_j(0,1)$, a contradiction. Thus, $\norm{x}_j \leq \norm{x}_k$ for all $x \in E$. Similarly, $\norm{x}_i \leq \norm{x}_k$ for all $x \in E$.
\end{proof}

\begin{remark}\label{rem:geom_ii}
It is worth adding that in contrast to Lemma~\ref{lem:geom_i}, we can replace closed balls with open ones in the statement of Lemma~\ref{lem:geom_ii} without any change to the assertion. 
\end{remark}

Now, let us give an abstract version of a definition which was first introduced in~\cite{BG20} in the case of the space $BV[0,1]$ of functions of bounded Jordan variation. As this new notion will play a key role in our further studies, we are going to spend some time investigating its various properties. 

\begin{definition}\label{def:equinormed}
Let $(E,\norm{\cdot}_E)$ be a normed space equipped with a family $\bigl\{\norm{\cdot}_i\bigr\}_{i \in I}$ of semi-norms satisfying conditions~\ref{i} and~\ref{ii}.
A non-empty subset $A$ of $E$ is called \emph{equinormed} (or more precisely, \emph{equinormed with respect to the family $\bigl\{\norm{\cdot}_i\bigr\}_{i \in I}$}) if for every $\varepsilon>0$ there exists an index $k \in I$ such that $\norm{x}_E\leq \varepsilon + \norm{x}_k$ for all $x \in A$.
\end{definition}

The following proposition gives a ``sequential'' version of Definition~\ref{def:equinormed}.

\begin{proposition}\label{prop:normseq}
Let $(E,\norm{\cdot}_E)$ be a normed space equipped with a family $\bigl\{\norm{\cdot}_i\bigr\}_{i \in I}$ of semi-norms satisfying conditions~\ref{i} and~\ref{ii}. A non-empty subset $A$ of $E$ is equinormed if and only if there exists an infinite sequence $(i_n)_{n\in \N}$ of elements of $I$ such that
\begin{enumerate}[label=\textup{(\roman*)}]
 \item\label{prop:normseq_i} $\norm{x}_E=\lim_{n\to \infty}\norm{x}_{i_n}$ uniformly on $A$,
 \item\label{prop:normseq_ii} $\norm{x}_{i_n}\leq \norm{x}_{i_{n+1}}$ for every $x\in E$ and $n\in \N$.
\end{enumerate}
\end{proposition}

\begin{proof}
Assume that the set $A$ is equinormed and fix $i_1\in I$ such that $\norm{x}_E\leq \norm{x}_{i_1}+1$ for every $x\in A$. Now, by \ref{ii}, we can choose an index $i_2 \in I$ such that $\norm{x}_{i_1}\leq \norm{x}_{i_2}$ for $x\in E$, and $\norm{x}_E\leq \norm{x}_{i_2}+\frac{1}{2}$ for $x\in A$. Repeating this process, we can choose a sequence $(i_n)_{n\in \N}$ of indices such that $\norm{x}_{i_n}\leq \norm{x}_{i_{n+1}}$ for $x\in E$, and $\norm{x}_E\leq \norm{x}_{i_n}+\frac{1}{n}$ for $x\in A$; here $n$ is an arbitrary positive integer. This, together with~\ref{i}, implies that $\norm{x}_{i_n}\leq \norm{x}_E\leq \norm{x}_{i_n}+\frac{1}{n}$ for every $n\in \N$ and $x\in A$. Hence, $\norm{x}_E=\lim_{n\to \infty}\norm{x}_{i_n}$ uniformly on $A$. 

The second implication follows from a straightforward observation that if $(i_n)_{n\in \N}$ is a sequence of elements of $I$ satisfying the above conditions~\ref{prop:normseq_i} and~\ref{prop:normseq_ii}, then for any $\varepsilon>0$ there exists $n\in \N$ such that $\sup_{x \in A}\abs[\big]{\norm{x}_E-\norm{x}_{i_n}}\leq \varepsilon$.  
\end{proof}

Now, let us take a look at several simple properties of equinormed sets.

\begin{proposition}\label{prop:properties_of_equinormed}
Let $(E,\norm{\cdot}_E)$ be a normed space equipped with a family $\bigl\{\norm{\cdot}_i\bigr\}_{i \in I}$ of semi-norms satisfying conditions~\ref{i} and~\ref{ii}. Then,
\begin{enumerate}[label=\textup{(\alph*)}]
  \item\label{prop:properties_of_equinormed_a} finite and non-empty subsets of $E$ are equinormed,
	\item\label{prop:properties_of_equinormed_b} a subset of an equinormed set is equinormed,
  \item\label{prop:properties_of_equinormed_c} finite unions of equinormed subsets of $E$ are equinormed,
	\item\label{prop:properties_of_equinormed_cc} arbitrary intersections of equinormed subsets of $E$ are equinormed, provided that they are non-empty, 
	\item\label{prop:properties_of_equinormed_d} a subset of $E$ is equinormed if and only if its closure \textup(in the norm $\norm{\cdot}_E$\textup) is,
	\item\label{prop:properties_of_equinormed_e} for every scalar $\lambda \neq 0$ the set $\lambda A$ is equinormed if and only if $A$ is, 
  \item\label{prop:properties_of_equinormed_f} for every $h>0$ the set $\bigcup_{\lambda \in [0,h]} \lambda A$ is equinormed if and only if $A$ is.
\end{enumerate}
\end{proposition}

\begin{proof}
The proofs of properties~\ref{prop:properties_of_equinormed_a}--\ref{prop:properties_of_equinormed_cc} are trivial. 

To prove~\ref{prop:properties_of_equinormed_d} observe that one implication follows from~\ref{prop:properties_of_equinormed_b}, while the other is a consequence of the passage to the limit in the inequality $\norm{x_n}_E \leq \varepsilon + \norm{x_n}_k$, where $(x_n)_{n \in \mathbb N}$ is a sequence of elements of the set $A$ convergent to some $x \in \overline{A}$ and the index $k$ is as in Definition~\ref{def:equinormed}.   

Property~\ref{prop:properties_of_equinormed_e} stems from the fact that $A=\frac{1}{\lambda}(\lambda A)$. 

Finally, if $A$ is equinormed, for a given $\varepsilon>0$ there is an index $k \in I$ such that $\norm{x}_E \leq \varepsilon/h + \norm{x}_k$ for $x \in A$. Then, for any $y \in \bigcup_{\lambda \in [0,h]} \lambda A$ we have $y=\lambda x$ for some $\lambda \in [0,h]$ and $x \in A$, and so $\norm{y}_E = \lambda \norm{x}_E \leq \lambda/h \cdot \varepsilon + \lambda \norm{x}_k \leq \varepsilon + \norm{y}_k$. This proves that the set $\bigcup_{\lambda \in [0,h]} \lambda A$ is equinormed. The other implication follows from the properties~\ref{prop:properties_of_equinormed_b} and~\ref{prop:properties_of_equinormed_e} and the fact that $hA \subseteq \bigcup_{\lambda \in [0,h]} \lambda A$. 
\end{proof}

In general, we cannot replace finite unions with infinite ones in Proposition~\ref{prop:properties_of_equinormed}~\ref{prop:properties_of_equinormed_c} as evidenced by the following example.

\begin{example}\label{ex:property_c_revised}
Consider the space $c_0$ endowed with the norm $\norm{\cdot}_{\infty}$ and the family $\{\norm{\cdot}_i\}_{i \in \mathbb N}$ of semi-norms given by the formulae $\norm{x}_i:=\max_{1\leq n \leq i}\abs{\xi_n}$, where $i \in \mathbb N$ and $x=(\xi_n)_{n \in \mathbb N} \in c_0$. (It is trivial to see that such a family of semi-norms satisfies conditions~\ref{i} and~\ref{ii}.) Clearly, each set $A_k:=\{e_k\}$, where $e_k$ denotes the $k$-th unit vector of $c_0$, as a singleton is equinormed in $c_0$. However, their union $A:=\bigcup_{k=1}^\infty A_k$ is not equinormed, because for each $i \in \mathbb N$ we have $\norm{e_{i+1}}_{\infty} = 1$ and $\norm{e_{i+1}}_i=0$.
\end{example}

\begin{remark}\label{rem:convex_hull_algebraic_sum_not_equinormed}
It is also worth noting that if $A,B$ are equinormed subsets of a normed space $(E,\norm{\cdot}_E)$ equipped with a family $\{\norm{\cdot}_i\}_{i \in I}$ of semi-norms satisfying conditions~\ref{i} and~\ref{ii}, then the sets $\conv A$ and $A+B$ may fail to be equinormed.

In order to prove this, similarly to Example~\ref{ex:property_c_revised}, let us consider the space $c_0$ endowed with the standard norm $\norm{\cdot}_{\infty}$ and the family $\{\norm{\cdot}_i\}_{i \in \mathbb N}$ of semi-norms given by the formulae $\norm{x}_i:=\max_{1\leq n \leq i}\abs{\xi_n}$, where $i \in \mathbb N$ and $x=(\xi_n)_{n \in \mathbb N} \in c_0$. Moreover, let $y_k:=e_1+\ldots+e_k$, where $e_j$ is the $j$-th unit vector of $c_0$, and set $A=\dset{y_k}{k \in \mathbb N}\cup \dset{-y_k}{k \in \mathbb N}$. The set $A$ is equinormed as for each $\varepsilon>0$ and $k \in \mathbb N$ we have $\norm{y_k}_{\infty} = 1 \leq \varepsilon + 1\leq \varepsilon + \norm{y_k}_1$. However, $\conv A$ fails to be equinormed as for each $i \in \mathbb N$ we have $\norm{\frac{1}{2}y_{i+1}-\frac{1}{2}y_i}_{\infty}=\frac{1}{2}$ and $\norm{\frac{1}{2}y_{i+1}-\frac{1}{2}y_i}_{i}=0$.

The claim concerning algebraic sums can be shown in a similar manner, if one considers the sets $A=\dset{y_k}{k \in \mathbb N}$ and $B=\dset{-y_k}{k \in \mathbb N}$. Let us also add that another example showing that algebraic sums of equinormed sets may fail to be equinormed (even if one of the summands is a singleton) was given (in a slightly different context of the space $BV[0,1]$ and equivariated sets) by K.~Czudek and published by his permission in~\cite{BG20}*{Example~2}.

To sum up this remark we can say that equinormed sets behave well when we apply set-theoretic operations to them, while they fail miserably when we switch to algebraic operations.
\end{remark}

\begin{remark}
It turns out also that equinormed sets do not have to be bounded. To see this it suffices to consider a very simple example of the space $(\mathbb R, \abs{\cdot})$ equipped with the family of semi-norms consisting of a single element $\abs{\cdot}$ and the set of positive integers $\mathbb N$.  
\end{remark}

As the notion of an equinormed set depends only on the metric structure of a space, it is natural to expect that they will be preserved under isometries. It turns out, however, that this is not the case, as the following example shows.

\begin{example}
Let us consider the space $(c, \norm{\cdot}_{\infty})$ equipped with the family $\{\norm{\cdot}_i\}_{i \in \mathbb N}$ of semi-norms which for $i \in \mathbb N$ and $x=(\xi_n)_{n \in \mathbb N} \in c$ are given by $\norm{x}_i:=\max_{1\leq n \leq i}\abs{\xi_n}$. Set $x_k:=e_1+\sum_{n=1}^\infty (1-\frac{1}{n})e_{k+n-1}$ and $A:=\dset{x_k}{k \in \mathbb N}$, where $e_j:=(0,\ldots,0,1,0,\ldots)$ denotes the $j$-th unit vector of $c$. Furthermore, let $f \colon c \to c$ be given by $f(x)=x-e_1$. The mapping $f$ is clearly an isometry, that is, $\norm{f(x)-f(y)}_\infty=\norm{x-y}_\infty$ for $x,y \in c$. Note also that $\norm{f(x)}_\infty=\norm{x}_\infty$ for $x \in A$, even though $f$ is not linear. It is not difficult to check that the set $A$ is equinormed. However, the image of $A$ under $f$ is not equinormed, because $\norm{f(x_i)}_\infty=1$ and $\norm{f(x_{i})}_i=0$ for every $i \in \mathbb N$.
\end{example}

Our aim in this section is to prove two new general compactness criteria in normed spaces. However, instead of providing a long proof of one of those theorems (and a sets of tips how to prove the other one), we decided to state several ``smaller'' results, as they are of independent interest. Let us start with the following proposition showing that precompact and relatively compact sets are necessarily equinormed. 

\begin{proposition}\label{prop:relatively_compact_implies_equinormed}
Let $(E,\norm{\cdot}_E)$ be a normed space equipped with a family $\bigl\{\norm{\cdot}_i\bigr\}_{i \in I}$ of semi-norms satisfying conditions~\ref{i} and~\ref{ii}. 
If $A$ is a non-empty and precompact \textup(or relatively compact\textup) subset of $E$, then it is equinormed.
\end{proposition}

\begin{proof}
Of course, we may assume that $A$ is precompact (totally bounded). Then, for any given $\varepsilon>0$ there exists a finite collection of points $x_1,\ldots,x_m \in E$ with the property that $A \subseteq \ball_E(x_1,\frac{1}{3}\varepsilon)\cup\ldots\cup\ball_E(x_m,\frac{1}{3}\varepsilon)$. Furthermore, in view of Proposition~\ref{prop:properties_of_equinormed}~\ref{prop:properties_of_equinormed_a}, we can find a semi-norm $\norm{\cdot}_k$ such that $\norm{x_j}_E\leq \frac{1}{3}\varepsilon + \norm{x_j}_k$ for every $j \in \{1,\ldots,m\}$. Now, take an arbitrary $x \in A$. As the set $A$ is covered by the closed balls $\ball_E(x_j,\frac{1}{3}\varepsilon)$, there is an index $i \in\{1,\ldots,m\}$ such that $\norm{x-x_i}_E \leq \frac{1}{3}\varepsilon$, and hence
\begin{align*}
 \norm{x}_E&\leq \norm{x-x_i}_E + \norm{x_i}_E \leq \tfrac{1}{3}\varepsilon+\tfrac{1}{3}\varepsilon + \norm{x_i}_k\\
 &\leq \tfrac{2}{3}\varepsilon + \norm{x_i-x}_k + \norm{x}_k \leq \tfrac{2}{3}\varepsilon + \norm{x_i-x}_E + \norm{x}_k \leq \varepsilon + \norm{x}_k.
\end{align*}  
This shows that $A$ is equinormed and ends the proof.
\end{proof}

Since relative compactness and precompactness carry over to algebraic sums of sets (in the setting of normed spaces, of course), from Proposition~\ref{prop:relatively_compact_implies_equinormed} we immediately get the following corollary.

\begin{corollary}\label{cor:relatively_compact_implies_equinormed}
Let $(E,\norm{\cdot}_E)$ be a normed space equipped with a family $\bigl\{\norm{\cdot}_i\bigr\}_{i \in I}$ of semi-norms satisfying conditions~\ref{i} and~\ref{ii}. 
If $A$ is a non-empty and precompact \textup(or relatively compact\textup) subset of $E$, then the sets $A-A$ and $x+A$, where $x \in E$, are equinormed.
\end{corollary}

Actually, we will see later that under a natural condition on the family of semi-norms $\bigl\{\norm{\cdot}_i\bigr\}_{i \in I}$, precompactness of a bounded set $A$ is  equivalent to the fact that $A-A$ is equinormed. However, the same claim is not true in the case of relatively compact sets as evidenced by the following example. The main reason for this example to work is the non-completeness of the underlying normed space.

\begin{example}
Let us consider the space $c_{00}$ of real sequences with only finitely many non-zero terms endowed with the norm $\norm{\cdot}_{\infty}$ and the semi-norms $\norm{x}_i:=\max_{1\leq n \leq i}\abs{\xi_n}$, where $i \in \mathbb N$ and $x=(\xi_n)_{n \in \mathbb N} \in c_{00}$. Moreover, let us set $x_k:=(1,\frac{1}{2},\ldots,\frac{1}{k},0,0,\ldots)$ for $k \in \mathbb N$, and $A:=\dset{x_k}{k \in \mathbb N}$. This set is clearly bounded.

We shall prove now that the set $A-A$ is equinormed. Fix an arbitrary $\varepsilon>0$ and choose $i \in \mathbb N$ such that $\frac{1}{i+1}\leq \varepsilon$. Furthermore, take $x_l,x_m \in A$ and assume that $l<m$. We have $x_m-x_l=(0,\ldots,0,\frac{1}{l+1},\ldots,\frac{1}{m},0,0,\ldots)$. If $i>l$, then $\norm{x_m-x_l}_{\infty}=\frac{1}{1+l}=\norm{x_m-x_l}_{i} \leq \varepsilon + \norm{x_m-x_l}_{i}$. On the other hand, if $i\leq l$, then $\norm{x_m-x_l}_{\infty}= \frac{1}{l+1} \leq \frac{1}{i+1} \leq \varepsilon + \norm{x_m-x_l}_{i}$. Therefore, the set $A-A$ is equinormed.

Note, however, that $A$ is not relatively compact in $c_{00}$ as the sole candidate for the limit of any subsequence of $(x_k)_{k \in \mathbb N}$ is $(1,\frac{1}{2},\frac{1}{3},\ldots)$ which does not belong to $c_{00}$.    
\end{example}

Now, we are going to prove the main results of this section characterizing relatively compact and precompact subsets of certain normed spaces.

\begin{theorem}\label{thm:compactness_ver1}
Let $(E,\norm{\cdot}_E)$ be a normed space equipped with a family $\bigl\{\norm{\cdot}_i\bigr\}_{i \in I}$ of semi-norms satisfying conditions~\ref{i} and~\ref{ii}. Furthermore, assume that 
\begin{enumerate}[label=\textbf{\textup{(A\arabic*)}}]
\setcounter{enumi}{2}
 \item\label{iii} for each bounded sequence $(x_n)_{n \in \mathbb N}$ of elements of $E$ there is a subsequence $(x_{n_k})_{k \in \mathbb N}$ and a point $x \in E$ such that $\norm{x-x_{n_k}}_i \to 0$ as $k \to +\infty$ for every $i \in I$.
\end{enumerate}
Then, a non-empty subset $A$ of $E$ is relatively compact if and only if it is bounded and for each $x \in \bigcap_{i \in I}\overline{A}^i$ the set $A-x$ is equinormed\textup; here $\overline{A}^i$ denotes the closure of $A$ with respect to the semi-norm $\norm{\cdot}_i$.
\end{theorem}

\begin{proof}
We will prove only the ``left'' implication. The ``right'' one is a consequence of Corollary~\ref{cor:relatively_compact_implies_equinormed} and the fact that relatively compact sets are bounded.

Let $(x_n)_{n \in \mathbb N}$ be an arbitrary sequence of elements of the set $A$. By~\ref{iii} there is a subsequence $(x_{n_k})_{n \in \mathbb N}$ of  $(x_n)_{n \in \mathbb N}$ and a point $x \in E$ such that $\norm{x-x_{n_k}}_i \to 0$ as $k \to +\infty$ for every index $i \in I$. We will show that $x_{n_k} \to x$ with respect to the norm convergence in $E$, which will prove that the set $A$ is relatively compact in $E$. So, let us fix $\varepsilon>0$. As $x \in \bigcap_{i \in I}\overline{A}^i$ and the set $A-x$ is equinormed, there is a semi-norm $\norm{\cdot}_j$ such that $\norm{y-x}_E \leq \frac{1}{2}\varepsilon + \norm{y-x}_j$ for every $y \in A$. Moreover, let $k_0 \in \mathbb N$ be such that $\norm{x-x_{n_k}}_j \leq \frac{1}{2}\varepsilon$ for all $k \geq k_0$. Then, for all $k \geq k_0$ we have $\norm{x-x_{n_k}}_E \leq \tfrac{1}{2}\varepsilon + \norm{x-x_{n_k}}_j \leq \varepsilon$,
which shows that the subsequence $(x_{n_k})_{n \in \mathbb N}$ converges to $x$ in $E$.
\end{proof}

Before, we move to the second theorem two remarks are in order. 

\begin{remark}\label{rem:comacptness_ver1_every_x}
Note that, by Corollary~\ref{cor:relatively_compact_implies_equinormed}, in the claim of Theorem~\ref{thm:compactness_ver1} we may equivalently assume that $A-x$ is equinormed for each $x \in E$.
\end{remark}

\begin{remark}
In the simplest case when the family of semi-norms $\bigl\{\norm{\cdot}_i\bigr\}_{i \in I}$ consists of a single semi-norm $\norm{\cdot}_E$, condition~\ref{iii} asserts that each bounded subset of $E$ is relatively compact, meaning that the space $(E,\norm{\cdot}_E)$ must be finite dimensional. In other words, the convergence postulated in~\ref{iii} is essentially weaker than the norm convergence in $(E,\norm{\cdot}_E)$ unless $\dim E<+\infty$.  
\end{remark}

\begin{theorem}\label{thm:compactness_ver2}
Let $(E,\norm{\cdot}_E)$ be a normed space equipped with a family $\bigl\{\norm{\cdot}_i\bigr\}_{i \in I}$ of semi-norms satisfying conditions~\ref{i} and~\ref{ii}. Furthermore, assume that 
\begin{enumerate}[label=\textbf{\textup{(A\arabic*)}}]
\setcounter{enumi}{3}
 \item\label{iv} for each bounded sequence $(x_n)_{n \in \mathbb N}$ of elements of $E$ there is a subsequence $(x_{n_k})_{k \in \mathbb N}$ which is Cauchy with respect to each semi-norm $\norm{\cdot}_i$.
\end{enumerate}
Then, a non-empty subset $A$ of $E$ is precompact if and only if it is bounded and the set $A-A$ is equinormed.
\end{theorem}

\begin{proof}
As in the case of Theorem~\ref{thm:compactness_ver1} we will only focus on proving the ``left'' implication.

Let $(x_n)_{n \in \mathbb N}$ be an arbitrary sequence of elements of the set $A$. By~\ref{iv} there is a subsequence $(x_{n_k})_{n \in \mathbb N}$ of  $(x_n)_{n \in \mathbb N}$ which is Cauchy with respect to each semi-norm $\norm{\cdot}_i$. We will show that this subsequence is in fact Cauchy with respect to the norm $\norm{\cdot}_E$, which will prove that $A$ is a precompact subset of $E$. So, let us fix $\varepsilon>0$. As the set $A-A$ is equinormed, there is a semi-norm $\norm{\cdot}_j$ such that $\norm{y-x}_E \leq \frac{1}{2}\varepsilon + \norm{y-x}_j$ for all $x,y \in A$. Moreover, let $k_0 \in \mathbb N$ be such that $\norm{x_{n_k}-x_{n_l}}_j \leq \frac{1}{2}\varepsilon$ for all $k,l \geq k_0$. Then, for all $k,l \geq k_0$ we have $\norm{x_{n_k}-x_{n_l}}_E \leq \tfrac{1}{2}\varepsilon + \norm{x_{n_k}-x_{n_l}}_j \leq \varepsilon$, which shows that the subsequence $(x_{n_k})_{n \in \mathbb N}$ is Cauchy.
\end{proof}

It is well-known that in the setting of complete metric spaces the classes of relatively compact subsets and precompact subsets coincide. Thus, a natural question arises: Can Theorems~\ref{thm:compactness_ver1} and \ref{thm:compactness_ver2} be used interchangeably when we additionally assume that the normed space in question is complete? A partial negative answer to this question is given by the following example.

\begin{example}\label{ex:c_0_counterexample}
Let us consider the Banach space $c_0$ equipped with the norm $\norm{\cdot}_{\infty}$ and the family of semi-norms $\norm{x}_i:=\max_{1\leq n\leq i}\abs{\xi_n}$, where $i \in \mathbb N$ and $x=(\xi_n)_{n \in \mathbb N} \in c_0$. Clearly, this family satisfies conditions~\ref{i} and~\ref{ii}. Moreover, given a bounded sequence $(x_k)_{k \in \mathbb N}$ in $c_0$, using the diagonal procedure, we can find its subsequence $(x_{k_l})_{l\in \mathbb N}$ which converges component-wise (to some bounded sequence). This means that $(x_{k_l})_{l \in \mathbb N}$ is a Cauchy sequence with respect to each semi-norm $\norm{\cdot}_i$. In other words, the family $\bigl\{\norm{\cdot}_i\bigr\}_{i \in \mathbb N}$ of semi-norms satisfies condition~\ref{iv}, and we can apply Theorem~\ref{thm:compactness_ver2} to characterize relatively compact subset of $c_0$.

However, if we consider the bounded sequence $(y_k)_{k \in \mathbb N}$, where $y_k:=e_1+\ldots+e_k$ and $e_j$ is the $j$-th unit vector of $c_0$, then each its subsequence is component-wise convergent to the sequence $(1,1,1,\ldots) \notin c_0$. Consequently, condition~\ref{iii} is not satisfied for the family of semi-norms $\bigl\{\norm{\cdot}_i\bigr\}_{i \in \mathbb N}$ defined above, meaning that Theorem~\ref{thm:compactness_ver1} cannot be applied in this case. 
\end{example}

\begin{question}
Unfortunately, we were not able to provide a complete answer to the question whether for Banach spaces Theorems~\ref{thm:compactness_ver1} and \ref{thm:compactness_ver2} can be used interchangeably. We do not know if it is possible to give an example of a Banach space $(E,\norm{\cdot}_E)$ with the following properties:
\begin{enumerate}[label=\textup{(\roman*)}]
 \item there is a family of semi-norms $\bigl\{\norm{\cdot}_i\bigr\}_{i \in I}$ on $E$ which satisfies conditions~\ref{i},~\ref{ii} and~\ref{iv},

 \item \emph{no} family of semi-norms $\bigl\{\norm{\cdot}_j\bigr\}_{j \in J}$ on $E$ satisfies  conditions~\ref{i},~\ref{ii} and~\ref{iii}.
\end{enumerate} 
\end{question}

\section{Compactness criterion in $l^p$-spaces}

One way to prove the classical and well-known compactness criterion in the space $l^p$ (in this paper it is stated as Theorem~\ref{thm:compactness_criterion_in_lp} below) is, first, to find a formula for the Hausdorff measure of non-compactness of a bounded subset of $l^p$, and then to deduce the criterion from it. (Readers interested in deriving the formula for the Hausdorff measure of non-compactness in $l^p$ are referred to e.g.~\cite{AKPRS}*{Section~1.1.9} or~\cite{BBK}*{Example~1.1.10}.) We will take another route and we will use one of our abstract results from the previous section, thus illustrating its applicability and as a by-product providing a new proof of this criterion.

\begin{theorem}\label{thm:compactness_criterion_in_lp}
Let $1 \leq p < +\infty$. A non-empty and bounded subset $A$ of $l^p$ is relatively compact if and only if for each $\varepsilon>0$ there exists an index $n \in \mathbb N$ such that $\sup_{x \in A} \bigl(\sum_{k=n+1}^\infty \abs{\xi_k}^p\bigr)^{1/p} \leq \varepsilon$, where $x=(\xi_k)_{k \in \mathbb N}$.
\end{theorem}

Before we proceed to the proof of Theorem~\ref{thm:compactness_criterion_in_lp} we need the following lemma.

\begin{lemma}\label{lem:inequality}
For $a,b \geq 0$ and $p\geq 1$ we have $(a+b)^p \leq a^p + pb(a+b)^{p-1}$.
\end{lemma}

\begin{proof}
Note that we may assume that $a>0$, since otherwise there is nothing to prove. Using differential calculus, it is easy to check that $(1+x)^p \leq 1 + px(1+x)^{p-1}$ for $x\geq 0$. Then, it suffices to substitute $x=b/a$ and get $(a+b)^p \leq a^p + pb(a+b)^{p-1}$.
\end{proof}

\begin{proof}[Proof of Theorem~\ref{thm:compactness_criterion_in_lp}]
For each $i \in \mathbb N$ let $\norm{x}_i:=\bigl(\sum_{k=1}^i \abs{\xi_k}^p\bigr)^{1/p}$, where $x=(\xi_k)_{k \in \mathbb N} \in l^p$. It is not difficult to check that the family  $\bigl\{\norm{\cdot}_i\bigr\}_{i \in \mathbb N}$ satisfies conditions~\ref{i} and~\ref{ii}. Also  condition~\ref{iii} is satisfied. Indeed, given a bounded sequence $(x_n)_{n \in \mathbb N}$ in $l^p$, we can find its subsequence $(x_{n_m})_{m \in \mathbb N}$ which is convergent component-wise to a sequence $x_\ast$. As for each $i \in \mathbb N$ and sufficiently large $n_m$ (depending on $i$) we have
\[
 \Biggl( \sum_{k=1}^i \abs{\xi_k^\ast}^p\Biggr)^{1/p} \leq \Biggl( \sum_{k=1}^i \abs{\xi_k^\ast-\xi_k^{n_m}}^p\Biggr)^{1/p} + \Biggl( \sum_{k=1}^i \abs{\xi_k^{n_m}}^p\Biggr)^{1/p} \leq 1 +M,
\]  
where $M>0$ is a fixed bound for the sequence $(x_n)_{n \in \mathbb N}$ and $\xi_k^\ast, \xi_k^{n_m}$ denote the terms of the sequences $x_\ast$ and $x_{n_m}$, respectively, we infer that $x_\ast \in l^p$. The fact that $\lim_{m \to \infty}\norm{x_{n_m}-x_\ast}_i = 0$ for every $i \in \mathbb N$ is obvious. Thus, by Theorem~\ref{thm:compactness_ver1} a non-empty and bounded subset $A$ of $l^p$ is relatively compact if and only if $A-y$ is equinormed for every $y \in l^p$ (cf.~Remark~\ref{rem:comacptness_ver1_every_x}). So, to end the proof of the compactness criterion we need to show that the above condition is equivalent with the condition stated in Theorem~\ref{thm:compactness_criterion_in_lp}.

Let us assume that for each $y \in  l^p$ the set $A-y$ is equinormed. Then, for $y=0$ and a given number $\varepsilon \in (0,1)$ there exists $N \in \mathbb N$ such that 
\[
 \norm{x}_{l^p} \leq \norm{x}_N + \frac{\varepsilon^p}{p(1+R)^{p-1}} 
\]
for any $x \in A$, where $R>0$ is such that $A\subseteq \ball_{l^p}(0,R)$. Using Lemma~\ref{lem:inequality}, for any $x \in A$, we have
\begin{align*}
 \sum_{k=1}^{N} \abs{\xi_k}^p + \sum_{k=N+1}^{\infty} \abs{\xi_k}^p & \leq \biggl(\norm{x}_N + \frac{\varepsilon^p}{p(1+R)^{p-1}}\biggr)^p\\
& \leq \norm{x}_N^p + \frac{\varepsilon^p}{(1+R)^{p-1}} \cdot \biggl(\norm{x}_N + \frac{\varepsilon^p}{p(1+R)^{p-1}}\biggr)^{p-1}\\
& \leq  \sum_{k=1}^{N} \abs{\xi_k}^p + \frac{\varepsilon^p}{(1+R)^{p-1}} \cdot (R+1)^{p-1}\\
& = \sum_{k=1}^{N} \abs{\xi_k}^p + \varepsilon^p.
\end{align*}
Thus, $\sup_{x \in A}\bigl(\sum_{k=N+1}^{\infty} \abs{\xi_k}^p\bigr)^{1/p} \leq \varepsilon$.

Now, let $y=(\eta_k)_{k \in \mathbb N} \in l^p$ be an arbitrary sequence. Moreover, let us fix $\varepsilon>0$ and let assume that we can find $n \in \mathbb N$ such that $\sup_{x \in A} \bigl(\sum_{k=n+1}^\infty \abs{\xi_k}^p\bigr)^{1/p} \leq \frac{1}{2}\varepsilon$. As the sequence $y$ is $p$-summable, there is an index $N\geq n$ such that $\bigl(\sum_{k=N+1}^\infty \abs{\eta_k}^p\bigr)^{1/p} \leq \frac{1}{2}\varepsilon$. Hence, for each $x \in A$ we have
\begin{align*}
 \norm{x-y}_{l^p} &= \Biggl(\sum_{k=1}^\infty \abs{\xi_k-\eta_k}^p\Biggr)^{1/p}\\
 & \leq \Biggl(\sum_{k=1}^N \abs{\xi_k-\eta_k}^p\Biggr)^{1/p} + \Biggl(\sum_{k=N+1}^\infty \abs{\xi_k-\eta_k}^p\Biggr)^{1/p}\\
 & \leq \norm{x-y}_N + \Biggl(\sum_{k=N+1}^\infty \abs{\xi_k}^p\Biggr)^{1/p}+\Biggl(\sum_{k=N+1}^\infty \abs{\eta_k}^p\Biggr)^{1/p}\\
 & \leq \varepsilon + \norm{x-y}_N,
\end{align*}
which shows that the set $A-y$ is equinormed. This completes the proof.
\end{proof}

\section{Compactness criterion in $C(X,\R)$}

Similar to what we did in the last section, we are now going to apply our abstract machinery to obtain the famous Arzel\`a--Ascoli theorem. This time, however, we are planning to use Theorem~\ref{thm:compactness_ver2}. 

Let us start with some general remarks concerning the space $C(X,\mathbb R)$. If $\card X<+\infty$, then $C(X,\R)$ is finite-dimensional and forms an equicontinuous family of functions. Consequently, this case is not very interesting, and we may assume that $\card X=+\infty$. It is well-known that compact metric spaces are separable, and so $X$ contains a countable and dense sequence (subset) $T := (t_j)_{j \in \mathbb N}$. Let us now consider the family $\bigl\{\norm{\cdot}_i\bigr\}_{i \in \mathbb N}$ of semi-norms on $C(X,\R)$, where
\begin{equation}\label{eq:semi_norms_on_C}
 \text{$\norm{x}_i:=\max_{1\leq j\leq i}\abs{x(t_j)}$ for each $i \in \mathbb N$ and $x \in C(X,\R)$.}
\end{equation}  
It is not difficult to check that this family satisfies conditions~\ref{i},~\ref{ii} and~\ref{iv}. (Condition~\ref{iv} is a consequence of the diagonal procedure.) However, it does not necessarily  satisfy condition~\ref{iii} as evidenced by the following example. 

\begin{example}\label{ex:equinormed}
Let $X=[0,1]$ and let us consider the Banach space $C[0,1]$ of all real-valued continuous functions defined on the interval $[0,1]$. Moreover, let the sequence $T$ consist of all dyadic rationals in $[0,1]$, that is, the numbers of the form $m/2^{l}$, where $l \in \mathbb N$, $m \in \mathbb N \cup\{0\}$ and $0\leq m \leq 2^l$, arranged in an arbitrary (but fixed) order. Suppose that for the bounded sequence $(x_n)_{n \in \mathbb N}$, where $x_n(t):=t^n$, there is a subsequence $(x_{n_k})_{k \in \mathbb N}$ and a continuous function $x \colon [0,1] \to \mathbb R$ such that $\norm{x_{n_k}-x}_i \to 0$ as $k \to +\infty$ for every $i \in \mathbb N$. As the convergence with respect to the semi-norms $\norm{\cdot}_i$ implies the pointwise convergence on $T$, we infer that $x(1)=1$ and $x(t)=0$ for $t \in T\setminus\{1\}$. But this is impossible, since $T$ is dense in $[0,1]$ and $x$ is continuous. 
\end{example}

\begin{remark}\label{rmk:equinormed}
Let us ponder over the above example a little bit more. By $A$ let us denote the subset of $C[0,1]$ consisting of all the functions $x_n$ from Example~\ref{ex:equinormed}. It is clear that $A$ is not equicontinuous. Furthermore, by direct calculations or Proposition~\ref{prop:equinormed_impies_equicontinuous}, which is proved below, one can check that the algebraic difference $A-A$ is not equinormed, although the set $A$ is (cf. Remark \ref{rem:convex_hull_algebraic_sum_not_equinormed} and \cite{BG20}*{Example~2}). Finally, observe that the sequence $(x_n)_{n\in \mathbb N}$ admits no convergent subsequence in $C[0,1]$.
\end{remark}

\begin{proposition}\label{prop:equicontinuous_impies_equinormed}
Let $X$ be a compact metric space with $\card X=+\infty$ and a dense sequence \textup(subset\textup) $T=(t_j)_{j \in \mathbb N}$. Moreover, let $\bigl\{\norm{\cdot}_i\bigr\}_{i \in \mathbb N}$ be the family of semi-norms on $C(X,\R)$ defined by~\eqref{eq:semi_norms_on_C}. If $A$ is a non-empty and equicontinuous subset of $C(X,\R)$, then it is equinormed.
\end{proposition}

\begin{proof}
Let $\varepsilon>0$ be fixed. Since the set $A$ is equicontinuous, there exists $\delta>0$ such that $\sup_{x \in A}\abs{x(t)-x(s)}\leq \varepsilon$ whenever $t,s \in X$ satisfy the condition $d(t,s)\leq \delta$. As $T$ is dense in $X$, we can find an index $k \in \mathbb N$ such that $X=\ball_X(t_1,\delta)\cup \ldots \cup \ball_X(t_k,\delta)$. Now, if we take any $x \in A$, we know that $\norm{x}_\infty = \abs{x(t_\ast)}$ for some $t_\ast \in X$ and that $t_\ast \in \ball_X(t_j,\delta)$ for some $j \in \{1,\ldots,k\}$. Then, $
 \norm{x}_\infty = \abs{x(t_\ast)} \leq \abs{x(t_j)}+\abs{x(t_\ast)-x(t_j)} \leq \norm{x}_k + \varepsilon$. This shows that the set $A$ is equinormed and ends the proof.
\end{proof}

In a similar vein to Corollary~\ref{cor:relatively_compact_implies_equinormed}, because the algebraic sum of two equicontinuous sets is equicontinuous from Proposition~\ref{prop:equicontinuous_impies_equinormed} we obtain the following result.

\begin{corollary}\label{cor:A_EC_impies_A_A_equinormed}
Let $X$ be a compact metric space with $\card X=+\infty$ and a dense sequence \textup(subset\textup) $T=(t_j)_{j \in \mathbb N}$. Moreover, let $\bigl\{\norm{\cdot}_i\bigr\}_{i \in \mathbb N}$ be the family of semi-norms on $C(X,\R)$ defined by~\eqref{eq:semi_norms_on_C}. If $A$ is a non-empty and equicontinuous subset of $C(X,\R)$, then the set $A-A$ is equinormed.
\end{corollary}

At this point it is obvious that the classes of equicontinuous and equinormed subsets of $C(X,\mathbb R)$ do not coincide (cf.~Example~\ref{ex:equinormed} and Remark~\ref{rmk:equinormed}). However, the above corollary raises another natural question: given a non-empty subset $A$ of $C(X,\R)$ is it true that $A$ is equicontinuous provided that $A-A$ is equinormed. The negative answer to this question is given by the following example.

\begin{example}
Once again let us consider the space $C[0,1]$ endowed with the family $\bigl\{\norm{\cdot}_i\bigr\}_{i \in \mathbb N}$ given by~\eqref{eq:semi_norms_on_C}, where $T=(t_j)_{j \in \mathbb N}$ is the sequence of all dyadic rational numbers contained in the interval $[0,1]$ arranged in a fixed order. Moreover, let $A$ be the set consisting of all affine functions $x_n \colon [0,1] \to \mathbb R$ of the form $x_n(t):=nt$, where $n \in \mathbb N$. As $1\in T$, we have $t_k=1$ for some $k \in \mathbb N$. Hence, for any $m,n \in \mathbb N$ we have
\[
 \norm{x_n-x_m}_\infty = \sup_{t \in [0,1]}\abs{nt-mt}=\abs{n-m}=t_k\abs{n-m}=\norm{x_n-x_m}_k.
\]
This shows that the set $A-A$ is equinormed. However, for any distinct $t,s \in [0,1]$, we have
\[
 \lim_{n \to \infty}\abs{x_n(t)-x_n(s)}=\lim_{n \to \infty}n\abs{t-s} = +\infty,
\]
which proves that $A$ is not an equicontinuous family of functions. 

As a closing remark, note that in the above reasoning we did not use the fact that $T$ consists of all dyadic rationals contained in $[0,1]$. All we needed was that $1$ was one of the terms of the sequence $T$.
\end{example}

The above example works because the set $A$ is unbounded in $C[0,1]$. It turns out, however, that adding the boundedness assumption renders the construction impossible. 

\begin{proposition}\label{prop:equinormed_impies_equicontinuous}
Let $X$ be a compact metric space with $\card X=+\infty$ and a dense sequence \textup(subset\textup) $T=(t_j)_{j \in \mathbb N}$. Moreover, let $\bigl\{\norm{\cdot}_i\bigr\}_{i \in \mathbb N}$ be the family of semi-norms on $C(X,\R)$ defined by~\eqref{eq:semi_norms_on_C}. If $A$ is a non-empty and bounded subset of $C(X,\R)$ such that $A-A$ is equinormed, then it is equicontinuous.
\end{proposition}

In the proof of Proposition~\ref{prop:equinormed_impies_equicontinuous} we will be using some ideas which can be traced back to Ambrosetti and his famous lemma concerning measures of non-compactness in the Banach space $C([a,b],E)$ of vector-valued continuous functions; for more details about this result see e.g. the original paper~\cite{Ambrosetti} or the monograph~\cite{BBK}*{Lemma~1.2.8}.

\begin{proof}[Proof of Proposition~\ref{prop:equinormed_impies_equicontinuous}]
Fix $\varepsilon>0$. As the set $A-A$ is equinormed, there exists an index $k \in \mathbb N$ such that $\norm{x-y}_\infty \leq \frac{1}{5}\varepsilon + \norm{x-y}_k$ for $x,y \in A$. Furthermore, let $M>0$ be such that $\norm{x}_\infty \leq M$ for $x \in A$ and choose a finite family of intervals $J_1,\ldots,J_n$ of length not exceeding $\frac{1}{5}\varepsilon$ such that $[-M,M]\subseteq J_1\cup\ldots \cup J_{n}$. Let $\Phi$ be the finite set of all functions $\varphi \colon \{1,\ldots,k\} \to \{1,\ldots,n\}$ and for each $\varphi \in \Phi$ set $A_\varphi:=\dset{x\in A}{\text{$x(t_j)\in J_{\varphi(j)}$ for every $j=1,\ldots,k$}}$. Clearly, some of the sets $A_\varphi$ may be empty, nonetheless $A=\bigcup_{\varphi \in \Phi}A_\varphi$. Set $\Psi:=\dset{\varphi\in \Phi}{A_\varphi\neq\emptyset}$ and for each $\varphi \in \Psi$ fix a function $y_\varphi \in A_\varphi$. Since $\dset{y_\varphi}{\varphi \in \Psi}$ is a finite subset of $C(X,\R)$, it is also equicontinuous. So we can find a $\delta>0$ such that $\sup_{\varphi \in \Psi}\abs{y_\varphi(t)-y_\varphi(s)} \leq \frac{1}{5}\varepsilon$ whenever $t,s \in X$ satisfy the condition $d(t,s)\leq \delta$. Now, take an arbitrary function $x \in A$ and any points $t,s \in X$ such that $d(t,s)\leq \delta$. Since $A=\bigcup_{\varphi \in \Psi}A_\varphi$, there is $\varphi \in \Psi$ such that $\norm{x-y_\varphi}_k\leq \frac{1}{5}\varepsilon$, and hence
\begin{align*}
 \abs{x(t)-x(s)} & \leq \abs{x(s)-y_\varphi(s)}+\abs{y_\varphi(s)-y_\varphi(t)}+\abs{y_\varphi(t)-x(t)}\\
 & \leq 2\norm{x-y_\varphi}_\infty + \abs{y_\varphi(s)-y_\varphi(t)}\\
 & \leq 2\norm{x-y_\varphi}_k + \tfrac{3}{5}\varepsilon\leq \varepsilon.
\end{align*}
This proves that the set $A$ is equicontinuous.
\end{proof}

All the above discussion leads to the well-known Arzel\`a--Ascoli compactness criterion.

\begin{theorem}
Let $X$ be a compact metric space. A non-empty and bounded subset of $C(X,\R)$ is relatively compact if and only if it is equicontinuous.
\end{theorem}

\begin{proof}
Let us consider two cases. If $\card X=+\infty$, then it suffices to apply Theorem~\ref{thm:compactness_ver2} together with Corollary~\ref{cor:A_EC_impies_A_A_equinormed} and Proposition~\ref{prop:equinormed_impies_equicontinuous} with the family of semi-norms defined by~\eqref{eq:semi_norms_on_C} and a fixed dense sequence (subset) $T=(t_j)_{j \in \mathbb N}$ of $X$.

On the other hand, if $\card X<+\infty$, then $C(X,\mathbb R)$ is linearly isometric with some $\mathbb R^n$ endowed with the maximum norm. So, the claim follows from the classical Bolzano--Weierstrass theorem, as in this case each non-empty subset of $C(X,\mathbb R)$ is equicontinuous.
\end{proof}

\begin{remark}
As we have already remarked although the notions of an equicontinuous and an equinormed set are closely related, they do not coincide. To see this more clearly let us consider some classical observations on sequences of functions (cf. Theorems 7.23--25 in \cite{Rud}). For simplicity, we consider the space $C[0,1]$ endowed with the family of semi-norms defined by~\eqref{eq:semi_norms_on_C} for a given dense sequence (subset) $T$ of $[0,1]$. Let $(x_n)_{n\in \mathbb N}$ be a sequence in $C[0,1]$.
\begin{enumerate}[label=(\alph*)]
\item\label{it:rem:a} If $(x_n)_{n\in \mathbb N}$ is uniformly convergent, then $\dset{x_n}{n\in \mathbb N}$ is equicontinuous. 

\item If $\dset{x_n}{n\in \mathbb N}$ is pointwise bounded and equicontinuous, then it is uniformly bounded (i.e. bounded as a subset of $C[0,1]$). 

\item If $(x_n)_{n\in \mathbb N}$ is pointwise convergent and the set $\dset{x_n}{n\in \mathbb N}$ is equicontinuous, then the sequence is uniformly convergent.

\item\label{it:rem:d} If $\dset{x_n}{n\in \mathbb N}$ is uniformly bounded and equicontinuous, then it contains a uniformly convergent subsequence.
\end{enumerate}
Let us replace the phrase ``$\dset{x_n}{n\in \mathbb N}$ is equicontinous'' with ``$\dset{x_n}{n\in \mathbb N}$ is equinormed'' in every item~\ref{it:rem:a}--\ref{it:rem:d}. After such a change the first statement is still true by Proposition~\ref{prop:equicontinuous_impies_equinormed}. The second one is true as well, since if the set $\dset{x_n}{n \in \mathbb N}$ is equinormed, there is an index $i\in \mathbb N$ for which the inequality $\norm{x_n}_{\infty} \leq 1+\norm{x_n}_i$ holds for all $n\in \mathbb N$. This, in turn, implies that for every $n \in \mathbb N$ we have $\norm{x_n}_{\infty} \leq 1+ \max_{1 \leq j \leq i} \abs{x_n(t_j)} \leq 1 + \max_{1 \leq j \leq i} \sup_{k \in \mathbb N} \abs{x_k(t_j)} < +\infty$. However, the last two modified assertions turn out to be false by Example~\ref{ex:equinormed} and Remark~\ref{rmk:equinormed}.
\end{remark}

\section{Compactness criterion in $\Phi BV[a,b]$}

So far we have shown how to obtain certain, nowadays classical, compactness criteria from our abstract results. In this section we are going to prove a compactness criterion in a space, where no such result is known. We will be dealing with the Banach space $\Phi BV[a,b]$ of functions of Schramm bounded variation. Although, at first glance this space may seem a bit artificial and of limited interest, it provides a unified approach to studying many classical spaces of functions of bounded variation and because of that is worth attention. Besides, what may come as a surprise is that each real-valued continuous function defined on a compact interval belongs to a certain space $\Phi BV[a,b]$. Our aim is  to examine relatively compact subsets of $\Phi BV[a,b]$. 

The definition of the Schramm variation (also known as the $\Phi$-variation) we are going to recall (see Definition~\ref{def:schramm_variation} below) is not the one Schramm introduced in his paper~\cite{schramm}. Sad to say, he did not pay sufficient attention to all the tiny details. For example, it is not clear whether Schramm considered finite or infinite collections of intervals, or maybe both. It is also shrouded in mystery if one is allowed to use degenerate intervals. A more precise approach to defining the Schramm variation was taken in the very nice monograph by Appell et al. (see~\cite{ABM}*{pp.~152--153}). Unfortunately, in that book it is also not clear whether degenerate intervals are allowed. So, before we will be able to proceed further, we have to stop for a while, take a look at dangers lurking in the shadows of the definition of the Schramm variation and try to dispel them. 

As already mentioned, there are two main factors we have to take into account: whether we consider finite or infinite families of intervals and whether we allow degenerate intervals (that is, those which consist of only one point) or not. This leads to four main cases and the relationship between them is described in the following proposition. (Since in this section, we will encounter some families consisting of intervals that are not necessarily non-overlapping, contrary to what some researchers suggest, we have decided not to omit the phrase ``closed and non-overlapping'' even if in some places it may look a little bit bulky. On the other hand, for simplicity's sake, we will assume that we work on the interval $[0,1]$. Finally, let us recall that we \emph{do} admit degenerate intervals, and that given a function $x$ and an interval $I:=[a,b]$ we write $x(I)$ for $x(b)-x(a)$.)

\begin{proposition}\label{prop:variation_various_definitions}
Let $(\varphi)_{n \in \mathbb N}$ be a fixed sequence of Young functions such that $\varphi_{k+1}(t)\leq \varphi_k(t)$ and $\sum_{n=1}^\infty \varphi_n(t)=+\infty$ for each $t >0$ and $k \in \mathbb N$. Given a function $x \colon [0,1] \to \mathbb R$ put
\begin{align*}
 A&:=\dset[\Bigg]{\sum_{n=1}^\infty \varphi_n\bigl(\abs{x(I_n)}\bigr)}{\parbox{22em}{$(I_n)_{n \in \mathbb N}$ is a sequence of closed, non-degenerate and non-overlapping subintervals of $[0,1]$}},\\[1mm]
 A^\ast&:=\dset[\Bigg]{\sum_{n=1}^\infty \varphi_n\bigl(2\abs{x(I_n)}\bigr)}{\parbox{22em}{$(I_n)_{n \in \mathbb N}$ is a sequence of closed, non-degenerate and non-overlapping subintervals of $[0,1]$}},\\[1mm]
 B&:=\dset[\Bigg]{\sum_{n=1}^\infty \varphi_n\bigl(\abs{x(I_n)}\bigr)}{\parbox{22em}{$(I_n)_{n \in \mathbb N}$ is a sequence of closed and non-overlapping subintervals}},\\[1mm]
 C&:=\dset[\Bigg]{\sum_{n=1}^k \varphi_n\bigl(\abs{x(I_n)}\bigr)}{\parbox{22em}{$I_1,\ldots,I_k$ is a finite collection of closed, non-degenerate and non-overlapping subintervals of $[0,1]$, where $k \in \mathbb N$}},\\[1mm]
D&:=\dset[\Bigg]{\sum_{n=1}^k \varphi_n\bigl(\abs{x(I_n)}\bigr)}{\parbox{22em}{$I_1,\ldots,I_k$ is a finite collection of closed and non-overlapping subintervals of $[0,1]$, where $k \in \mathbb N$}}.
\end{align*}
If $\alpha:=\sup A$, $\alpha^\ast:=\sup A^\ast$, $\beta:=\sup B$, $\gamma:=\sup C$ and $\delta:=\sup D$, then $\alpha \leq \beta =\gamma = \delta \leq \alpha^\ast$.
\end{proposition}

\begin{proof}
The inequalities $\alpha \leq \beta$ and $\gamma \leq \delta \leq \beta$ are obvious. (Note that the last one is true because in the sequences appearing in $B$ we allow degenerate intervals.)

Now, we are going to show that $\beta \leq \gamma$. Let $(I_n)_{n \in \mathbb N}$ be an arbitrary sequence of closed and non-overlapping subintervals of $[0,1]$. Fix a positive integer $k$. If $\sum_{n=1}^k \varphi_n(\abs{x(I_n)})=0$, then clearly $\sum_{n=1}^k \varphi_n(\abs{x(I_n)})\leq \gamma$. Thus, we may assume that $\sum_{n=1}^k \varphi_n(\abs{x(I_n)})>0$. Similarly, if $\varphi_j(\abs{x(I_j)})>0$ for every $j=1,\ldots,k$, then each interval $I_j$ is non-degenerate, and consequently $\sum_{n=1}^k \varphi_n(\abs{x(I_n)})\leq \gamma$. So, it suffices to consider the situation, when $\varphi_j(\abs{x(I_j)})=0$ for some $j \in \{1,\ldots,k\}$. Then, using the pointwise monotonicity of the sequence $(\varphi_n)_{n \in \mathbb N}$ of Young functions, we obtain
\[
 \sum_{n=1}^k \varphi_n\bigl(\abs{x(I_n)}\bigr) \leq \sum_{n=1}^{j-1} \varphi_n\bigl(\abs{x(I_n)}\bigr) + \sum_{n=j}^{k} \varphi_n\bigl(\abs{x(I_{n+1})}\bigr);
\]
(we use the convention that when the upper summation index is smaller than the lower one, then the sum in question is defined to be zero). Repeating this ``elimination of zeros'' process (at most $k-1$ times), we finally arrive at the estimate
\[
 \sum_{n=1}^k \varphi_n\bigl(\abs{x(I_n)}\bigr) \leq \sum_{n=1}^l \varphi_n\bigl(\abs{x(J_n)}\bigr),
\] 
where $J_1,\ldots,J_l$ is a finite collection of closed, non-degenerate and non-overlapping subintervals of $[0,1]$. This proves that $\sum_{n=1}^k \varphi_n(\abs{x(I_n)})\leq \gamma$. As the number $k \in \mathbb N$ was arbitrary we conclude that $\beta \leq \gamma$.

Now, it remains to show that $\gamma \leq \alpha^\ast$. Our proof will mimic the approach presented in~\cite{ABM}*{pp.~129--130}, which was used to establish a similar (but slightly different) result for the Waterman variation. Clearly, we may assume that $\alpha^\ast<+\infty$. Furthermore, fix any finite collection $I_1,\ldots,I_k$ of closed, non-degenerate and non-overlapping subintervals of $[0,1]$. If $\bigcup_{n=1}^{k} I_n \neq [0,1]$, then taking any sequence $(J_n)_{n \in \mathbb N}$ of closed, non-degenerate and non-overlapping subintervals of  $[0,1]\setminus \bigcup_{n=1}^{k} I_n$ and using the convexity of the Young functions $\varphi_n$ and the fact that $\varphi_n(0)=0$, we obtain
\begin{align*}
  \sum_{n=1}^k \varphi_n\bigl(\abs{x(I_n)}\bigr) \leq  \frac{1}{2}\sum_{n=1}^k \varphi_n\bigl(2\abs{x(I_n)}\bigr) +  \frac{1}{2}\sum_{n=k+1}^\infty \varphi_n\bigl(2\abs{x(J_{n-k})}\bigr) \leq \frac{1}{2}\alpha^\ast.
\end{align*}

Now, suppose that $\bigcup_{n=1}^{k} I_n=[0,1]$ and choose any two sequences $(J_n^1)_{n \in \mathbb N}$ and $(J_n^2)_{n \in \mathbb N}$ of closed, non-degenerate and non-overlapping subintervals of $I_k:=[a_k,b_k]$ with $J_1^1:=[a_k, \frac{1}{2}(a_k+b_k)]$ and $J_1^2:=[\frac{1}{2}(a_k+b_k), b_k]$. (Let us underline that we do not assume that the intervals belonging to the sequence $(J_n^1)_{n \in \mathbb N}$ are non-overlapping with those belonging to $(J_n^2)_{n \in \mathbb N}$.) Then, using properties of the functions $\varphi_n$, we get
\begin{align*}
 \sum_{n=1}^k \varphi_n\bigl(\abs{x(I_n)}\bigr)	&\leq  \sum_{n=1}^{k-1} \varphi_n\bigl(\abs{x(I_n)}\bigr) + \varphi_k\bigl(\abs{x(J_1^1)} + \abs{x(J_1^2)}\bigr)\\
	& \leq  \sum_{n=1}^{k-1} \varphi_n\bigl(2\abs{x(I_n)}\bigr) + \frac{1}{2}\varphi_k\bigl(2\abs{x(J_1^1)}\bigr) +  \frac{1}{2}\varphi_k\bigl(2\abs{x(J_1^2)}\bigr)\\
	& \leq \frac{1}{2}\Biggl(\sum_{n=1}^{k-1} \varphi_n\bigl(2\abs{x(I_n)}\bigr) + \sum_{n=1}^{\infty} \varphi_{n+k-1}\bigl(2\abs{x(J_n^1)}\bigr) \Biggr)\\
	&\qquad + \frac{1}{2}\Biggl(\sum_{n=1}^{k-1} \varphi_n\bigl(2\abs{x(I_n)}\bigr) + \sum_{n=1}^{\infty} \varphi_{n+k-1}\bigl(2\abs{x(J_n^2)}\bigr) \Biggr)\\
	& \leq \frac{1}{2}\alpha^\ast + \frac{1}{2}\alpha^\ast=\alpha^\ast.
\end{align*}

We have thus shown that regardless of whether the union  $\bigcup_{n=1}^{k} I_n$ coincides with the interval $[0,1]$ or not, we have $\sum_{n=1}^k \varphi_n\bigl(\abs{x(I_n)}\bigr) \leq \alpha^\ast$. This shows that $\gamma \leq \alpha^\ast$ and completes the proof.
\end{proof}

The following example illustrates that in general both inequalities in Proposition~\ref{prop:variation_various_definitions} may be strict.

\begin{example}\label{ex:schramm_variation_1}
For $n \in \mathbb N$ consider the Young functions $\varphi_n \colon [0,+\infty) \to [0,+\infty)$ given by the formulas $\varphi_1(t)=10t$ and $\varphi_2(t)=\varphi_3(t)=\ldots=t$. Moreover, let $x\colon [0,1] \to \mathbb R$ be defined as
\[
 x(t)=\begin{cases}
        0, &\text{if $t=0$},\\
				t+1, &\text{if $t \in (0,1)$},\\
				3, & \text{if $t=1$}.
				\end{cases}
\]

Let $I_1,\ldots,I_k$ be an arbitrary finite collection of closed, non-degenerate and non-overlapping subintervals of $[0,1]$. If $k=1$, then, obviously, $\sum_{n=1}^k \varphi_n(\abs{x(I_n)})\leq 30$. If, on the other hand, $k\geq 2$, we have $\sum_{n=1}^k \varphi_n(\abs{x(I_n)})\leq 23$, as the points $0$ and $1$ cannot belong to the same interval $I_n$. (Note that as a by-product of this estimate we also get that $\alpha \leq 23$.) Finally, noting that for the interval $[0,1]$ we have $\varphi_1(\abs{x(1)-x(0)})=30$, we deduce that $\beta=\gamma=\delta=30$.

Now, let us estimate $\alpha^\ast$. Let $c_0:=0$ and $c_k:=\sum_{n=1}^{k} 2^{-k}$ for $k \in \mathbb N$. Then,
\[
 \sum_{n=1}^\infty \varphi_n\bigl(2\abs{x(c_n)-x(c_{n-1})}\bigr)\geq 20\abs{x(\tfrac{1}{2})-x(0)} + 2\abs{x(\tfrac{3}{4})-x(\tfrac{1}{2})} = \tfrac{61}{2}.
\]
Thus, $\alpha \leq 23 < 30 = \beta=\gamma=\delta < \tfrac{61}{2} \leq \alpha^\ast$.
\end{example} 

\subsection{Functions of bounded Schramm variation}  
After this lengthy technical introduction, we are finally ready to provide the definition of the Schramm variation. We have chosen to take the  approach ``D'' as it seems the simplest and most flexible one.

\begin{definition}\label{def:schramm_variation}
Let $x$ be a real-valued function defined on $[0,1]$ and let $(\varphi_n)_{n \in \mathbb N}$ be a sequence of Young functions such that $\varphi_{k+1}(t)\leq \varphi_k(t)$ and $\sum_{n=1}^\infty \varphi_n(t)=+\infty$ for each $t >0$ and $k \in \mathbb N$. The (possibly infinite) number
\[
 \var_{\Phi} x :=\sup \sum_{n=1}^{k}\varphi_n\big(\abs{x(I_n)}\big),
\]
where the supremum is taken over all finite collections $I_1,\ldots,I_k$ of closed and non-overlapping subintervals of $[0,1]$, $k \in \mathbb N$, is called the \emph{Schramm variation} (or, the \emph{$\Phi$-variation}) of the function $x$ over $[0,1]$.

A function $x \colon [0,1] \to \mathbb R$ is said to be of \emph{bounded Schramm variation} if there exists a number $\lambda>0$ such that $\var_\Phi(\lambda x)<+\infty$.
\end{definition}

\begin{remark}
Observe that due to Proposition~\ref{prop:variation_various_definitions} the class of all real-valued functions defined on the interval $[0,1]$ with bounded Schramm variation, that is, the set
\[
\Phi BV[0,1]:=\dset{x \colon [0,1] \to \mathbb R}{\text{$\var_\Phi (\lambda x)<+\infty$ for some $\lambda>0$}},
\]
is independent of the choice of the approach. In other words, in the definition of a function of bounded Schramm variation we can consider variations taken with respect to finite or/and infinite sequences of intervals which are either degenerate or non-degenerate, and eventually we will end up with the same class of functions. It can be shown that $\Phi BV[0,1]$ is a linear space. 
\end{remark}

\begin{remark}\label{rem:variations}
The Schramm variation generalizes many known and classical variations. Indeed, if for $n \in \mathbb N$ we define the Young functions $\varphi_n \colon [0,+\infty)\to [0,+\infty)$ as, respectively, $\varphi_n(t)=t$, $\varphi_n(t)=t^p$ with $p \in [1,+\infty)$, $\varphi_n(t)=\varphi(t)$, where $\varphi$ is a fixed $\varphi$-function, and $\varphi_n(t)=\lambda_nt$, where the non-increasing sequence of positive numbers $(\lambda_n)_{n \in \mathbb N}$ is such that $\sum_{n=1}^\infty \lambda_n=+\infty$, then from Definition~\ref{def:schramm_variation} we recover the definitions of variation in the sense of Jordan, Wiener, Young and finally Waterman. (For more information on classical and non-classical variations we refer the reader to the monograph~\cite{ABM}.)  
\end{remark}

Our main goal in this part of the paper is to provide a compactness criterion in the space $\Phi BV[0,1]$. To this end we are going to apply our abstract results from Section~\ref{sec:abstract_compactness_criterion}. So, we need to define a suitable family of semi-norms on $\Phi BV[0,1]$, and this is exactly what we are going to do now.

\begin{notation}
Fix a sequence $(\varphi_n)_{n \in \mathbb N}$ of Young functions such that $\varphi_{k+1}(t)\leq \varphi_k(t)$ and $\sum_{n=1}^\infty \varphi_n(t)=+\infty$ for each $t >0$ and $k \in \mathbb N$. Let $\mathcal P$ be the family of all closed subintervals of $[0,1]$ and let $\mathcal J \subseteq \mathcal P$; we set
\[
 V_{\mathcal J}(x):=\sup \sum_{n=1}^{k}\varphi_n\big(\abs{x(I_n)}\big),
\]
where the supremum is taken over all finite collections $I_1,\ldots,I_k$ of non-overlapping intervals such that $I_n \in \mathcal J$ for $n=1,\ldots,k$, $k \in \mathbb N$. Note that in the special case when $\mathcal J=\mathcal P$, we have $V_{\mathcal P}(x)=\var_{\Phi}(x)$. Finally, we define
\begin{equation}\label{eq:semi-norm}
 \abs{x}_{\mathcal J}:=\inf\dset[\big]{\lambda>0}{V_{\mathcal J}(\tfrac{x}{\lambda})\leq 1}
\end{equation} 
and
\begin{equation}\label{eq:norm}
 \norm{x}_{\mathcal J}:=\abs{x(0)}+ \abs{x}_{\mathcal J}.
\end{equation} 
In accordance with many text concerning functions of bounded variation, in the sequel we will write $\abs{x}_{\Phi}$ and $\norm{x}_{\Phi}$ instead of $\abs{x}_{\mathcal P}$ and $\norm{x}_{\mathcal P}$, respectively.
\end{notation}   

It turns out that for each family $\mathcal J$ the formula~\eqref{eq:norm} defines a semi-norm on $\Phi BV[0,1]$. The proof of this fact runs along the same lines as the proof of~\cite{schramm}*{Theorem~2.2} or~\cite{ABM}*{Proposition~2.44} with the slight difference  that we do not assume that the considered functions of bounded Schramm variation vanish at $t=0$. Hence, we have decided to omit it. 

\begin{proposition}
Let $(\varphi_n)_{n \in \mathbb N}$ be a fixed sequence of Young functions such that $\varphi_{k+1}(t)\leq \varphi_k(t)$ and $\sum_{n=1}^\infty \varphi_n(t)=+\infty$ for each $t >0$ and $k \in \mathbb N$. Moreover, let $\mathcal J$ be a given family of closed subintervals of $[0,1]$. Then, the formula~\eqref{eq:norm} defines a semi-norm on the space $\Phi BV[0,1]$. Furthermore, $\norm{\cdot}_{\Phi}$ is a norm. 
\end{proposition}

\begin{remark}\label{rem:norm_on_PBV}
It is not difficult to show that each function $x \in \Phi BV[0,1]$ is bounded and that $\norm{x}_{\infty}\leq c_\Phi\norm{x}_{\Phi}$ with $c_\Phi:=\max\{1,\varphi_1^{-1}(1)\}$. Indeed, if $t \in [0,1]$ and $\lambda>0$ is such that $\var_{\Phi}(\lambda^{-1}x)\leq 1$, then
\begin{align*}
 \lambda^{-1}\abs{x(t)}&\leq \lambda^{-1}\abs{x(0)} + \varphi_1^{-1}\bigl(\varphi_1\bigl(\lambda^{-1}\abs{x(t)-x(0)}\bigr)\bigr)\\
&\leq \lambda^{-1}\abs{x(0)}+\varphi_1^{-1}(\var_\Phi (\lambda^{-1}x))\\
&\leq \lambda^{-1}\abs{x(0)} + \varphi_1^{-1}(1).
\end{align*}
Hence, $\lambda^{-1}\norm{x}_{\infty} \leq \lambda^{-1}\abs{x(0)} + \varphi_1^{-1}(1)$, and consequently
\[
 \frac{\norm{x}_{\infty} - \abs{x(0)}}{\varphi_1^{-1}(1)} \leq \lambda.
\]
This, in turn, implies that $\norm{x}_{\infty} \leq \abs{x(0)} + \varphi_1^{-1}(1)\abs{x}_\Phi \leq c_\Phi\norm{x}_\Phi$.
\end{remark}

We end this subsection with a technical result which will come in handy when proving that the family of semi-norms $\{\norm{\cdot}_{\mathcal J}\}_{\mathcal J \subseteq \mathcal P}$ on $\Phi BV[0,1]$ satisfies conditions~\ref{i}--\ref{iii}.

\begin{proposition}\label{prop:equivalent_norm_F}
Let $(\varphi_n)_{n \in \mathbb N}$ be a fixed sequence of Young functions such that $\varphi_{k+1}(t)\leq \varphi_k(t)$ and $\sum_{n=1}^\infty \varphi_n(t)=+\infty$ for each $t >0$ and $k \in \mathbb N$. Moreover, let $\mathcal J$ be a given family of closed subintervals of $[0,1]$. Then, for every $x \in \Phi BV[0,1]$, we have
\begin{equation}\label{eq:equivalent_norm_F_formula}
 \abs{x}_{\mathcal J}=\sup\biggl(\inf\dset[\bigg]{\lambda>0}{ \sum_{n=1}^{k}\varphi_n\big(\abs{\lambda^{-1}x(I_n)}\big)\leq 1}\biggr),
\end{equation}
where the supremum is taken over all finite collections $I_1,\ldots,I_k$ of non-overlapping intervals such that $I_n \in \mathcal J$ for $n=1,\ldots,k$, $k \in \mathbb N$.
\end{proposition}

\begin{proof}
Let $x \in \Phi BV[0,1]$ be fixed and by $\scal{x}_{\mathcal J}$ let us denote the right-hand side of the formula~\eqref{eq:equivalent_norm_F_formula}. Moreover, let $I_1,\ldots,I_k$ be a finite collection of non-overlapping intervals such that $I_n \in \mathcal J$ for $n=1,\ldots,k$. Then, by definition of $V_{\mathcal J}$, for every $\lambda >0$ we have
\[
 \sum_{n=1}^{k}\varphi_n\big(\abs{\lambda^{-1}x(I_n)}\big) \leq V_{\mathcal J}(\lambda^{-1}x).
\]
This implies that  
\[
\inf\dset[\bigg]{\lambda>0}{ \sum_{n=1}^{k}\varphi_n\big(\abs{\lambda^{-1}x(I_n)}\big)\leq 1} \leq \inf\dset[\big]{\lambda>0}{V_{\mathcal J}(\lambda^{-1}x)\leq 1} =\abs{x}_{\mathcal J}.
\]
As the collection $I_1,\ldots,I_k$ was arbitrary, we obtain $\scal{x}_{\mathcal J}\leq \abs{x}_{\mathcal J}$.

To prove the opposite inequality, let us note that for every finite collection  $I_1,\ldots,I_k$ of non-overlapping intervals such that $I_n \in \mathcal J$ for $n=1,\ldots,k$ and every $\varepsilon>0$ there exists a number $\lambda$ such that $0<\lambda<\varepsilon+\scal{x}_{\mathcal J}$ and $\sum_{n=1}^{k}\varphi_n\big(\abs{\lambda^{-1}x(I_n)}\big)\leq 1$. As 
\[
\sum_{n=1}^{k}\varphi_n\biggl(\frac{\abs{x(I_n)}}{\varepsilon+\scal{x}_{\mathcal J}}\biggr) \leq \sum_{n=1}^{k}\varphi_n\biggl(\frac{\abs{x(I_n)}}{\lambda}\biggr) \leq 1,
\]
we conclude that
\[
V_{\mathcal J}\biggl(\frac{x}{\varepsilon+\scal{x}_{\mathcal J}}\biggr)\leq 1
\]
for every $\varepsilon>0$. And so $\abs{x}_{\mathcal J}\leq \varepsilon+\scal{x}_{\mathcal J}$. To end the proof it suffices to observe that the number $\varepsilon>0$ is arbitrary.
\end{proof}

\subsection{Relatively compact subsets of $\Phi BV[0,1]$}
As the title of this subsection suggests, we are now going to characterize relatively compact subsets of the space $\Phi BV[0,1]$. We begin with introducing our setting.

\begin{proposition}\label{prop:semi_norm_F}
Let $(\varphi_n)_{n \in \mathbb N}$ be a fixed sequence of Young functions such that the functions $\varphi_{n+1}-\varphi_n$ are non-increasing on $[0,+\infty)$ for $n \in \mathbb N$ and $\sum_{n=1}^\infty \varphi_n(t)=+\infty$ for each $t >0$. Then, the family of semi-norms $\{\norm{\cdot}_{\mathcal J}\}_{\mathcal J \subseteq \mathcal P}$ on $(\Phi BV[0,1], \norm{\cdot}_{\Phi})$ indexed by finite subfamilies $\mathcal J$ of $\mathcal P$, where $\norm{\cdot}_{\mathcal J}$ is given by the formula~\eqref{eq:norm}, satisfies conditions~\ref{i}--\ref{iii}.      
\end{proposition}

\begin{proof}
The proof falls naturally into three steps. Note that if $\mathcal J_1 \subseteq \mathcal J_2 \subseteq \mathcal P$, then $\abs{x}_{\mathcal J_1}\leq \abs{x}_{\mathcal J_2}$ for every $x \in \Phi BV[0,1]$. This implies that the family $\{\norm{\cdot}_{\mathcal J}\}_{\mathcal J \subseteq \mathcal P}$ of semi-norms indexed by finite subfamilies $\mathcal J$ of $\mathcal P$ satisfies condition~\ref{ii}, because for two given indices (finite families) $\mathcal J_1,\mathcal J_2 \subseteq  \mathcal P$ we can take $\mathcal J_1 \cup \mathcal J_2$ as the succeeding index (finite family).

To show that~\ref{i} is also satisfied fix $x \in \Phi BV[0,1]$. Furthermore, set
\[
 \scal{x}_{\mathcal L}:=\sup\biggl(\inf\dset[\bigg]{\lambda>0}{ \sum_{n=1}^{k}\varphi_n\big(\abs{\lambda^{-1}x(I_n)}\big)\leq 1}\biggr),
\]
where the supremum is taken over all finite collections $I_1,\ldots,I_k$ of non-overlapping intervals such that $I_n \in \mathcal L$ for $n=1,\ldots,k$. In particular, if $I_1,\ldots,I_k$ are non-overlapping and closed subintervals of $[0,1]$, then
\[
 \inf\dset[\bigg]{\lambda>0}{ \sum_{n=1}^{k}\varphi_n\big(\abs{\lambda^{-1}x(I_n)}\big)\leq 1} \leq \scal{x}_{\mathcal J_\ast},
\]
where $\mathcal J_\ast:=\{I_1,\ldots,I_k\}$. And so $\scal{x}_{\mathcal P}\leq \sup_{\mathcal J \subseteq \mathcal P}\scal{x}_{\mathcal J}$, where the supremum is taken over all finite subfamilies $\mathcal J$ of $\mathcal P$. Then, by Proposition~\ref{prop:equivalent_norm_F} we have $\abs{x}_{\Phi}=\abs{x}_{\mathcal P}=\scal{x}_{\mathcal P} \leq \sup_{\mathcal J \subseteq \mathcal P} \scal{x}_{\mathcal J} = \sup_{\mathcal J \subseteq \mathcal P} \abs{x}_{\mathcal J} \leq \abs{x}_{\mathcal P}=\abs{x}_{\Phi}$; here, as before, the suprema are taken over all finite subfamilies $\mathcal J$ of $\mathcal P$. From the above equality it now follows easily that condition~\ref{i} is satisfied.

It remains now to prove~\ref{iii}. Let $(x_n)_{n \in \mathbb N}$ be an arbitrary bounded sequence in $\Phi BV[0,1]$. Then, there exist two constants $c>0$ and $M>0$ such that $\sup_{n \in \mathbb N}\norm{cx_n}_{\infty} \leq M$ and $\sup_{n \in \mathbb N}\var_\Phi(cx_n)\leq M$ (cf.~Remark~\ref{rem:norm_on_PBV}). Applying  Helly's selection theorem for the Schramm variation (see~\cite{ABM}*{Theorem~2.49} or~\cite{schramm}*{Theorem~2.8}), we deduce that there is a subsequence $(x_{n_k})_{k \in \mathbb N}$ of $(x_n)_{n \in \mathbb N}$ which is pointwise convergent to some $x \in \Phi BV[0,1]$ on $[0,1]$.   

Fix a family $\mathcal J:=\{I_1,\ldots,I_m\} \subseteq \mathcal P$ with each interval $I_j$ of the form $[a_j,b_j]$ and an arbitrary positive number $\varepsilon$.  In view of the pointwise convergence of the subsequence $(x_{n_k})_{k \in \mathbb N}$ we can find an index $k_0 \in \mathbb N$ such that $\abs{x_{n_k}(0)-x(0)}\leq \varepsilon$ and $\abs{x_{n_k}(t_j)-x(t_j)}\leq \varepsilon\varphi_1^{-1}(m^{-1})$ for all $k \geq k_0$ and $j=1,\ldots,m$; here $t_j \in \{a_j,b_j\}$. If $J_1,\ldots,J_l$ is a finite collection of non-overlapping subintervals of $[0,1]$ such that $J_j:=[c_j,d_j] \in \mathcal J$, then for $k \geq k_0$ we have
\begin{align*}
&\sum_{j=1}^l \varphi_j\bigl(\tfrac{1}{2}\varepsilon^{-1} \abs{(x_{n_k}-x)(J_j)}\bigr)\\
&\quad=\sum_{j=1}^l \varphi_j\bigl(\tfrac{1}{2}\varepsilon^{-1}\abs{(x_{n_k}-x)(d_j)-(x_{n_k}-x)(c_j)}\bigr)\\
&\quad \leq \frac{1}{2}\sum_{j=1}^l \varphi_j\bigl(\varepsilon^{-1}\abs{x_{n_k}(d_j)-x(d_j)}\bigr) +  \frac{1}{2}\sum_{j=1}^l \varphi_j\bigl(\varepsilon^{-1}\abs{x_{n_k}(c_j)-x(c_j)}\bigr)\\
&\quad \leq \sum_{j=1}^l \varphi_j(\varphi^{-1}_1(m^{-1})) \leq l \varphi_1(\varphi^{-1}_1(m^{-1})) = lm^{-1}\leq 1;
\end{align*}
the first inequality in the last line holds because $\dset{c_j,d_j}{j=1,\ldots,l} \subseteq \dset{a_j,b_j}{j=1,\ldots,m}$ and the last one follows from the fact that the number $l$ cannot exceed the cardinality of $\mathcal J$. Thus, $V_{\mathcal J}\bigl(\tfrac{1}{2}\varepsilon^{-1}(x_{n_k}-x)\bigr)\leq 1$ for $k \geq k_0$. Consequently, for $k \geq k_0$ we have
\[
 \norm{x_{n_k}-x}_{\mathcal J} = \abs{x_{n_k}(0)-x(0)}+\abs{x_{n_k}-x}_{\mathcal J}\leq 3\varepsilon.
\]
This shows that condition~\ref{iii} is satisfied.  
\end{proof}

\begin{remark}\label{rem:explanation_why_add_assumption}
We feel that we need to explain why (in contrast to earlier results) in Proposition~\ref{prop:semi_norm_F} we strengthened the assumption that $\varphi_{n+1}(t)\leq \varphi_n(t)$ for $n \in \mathbb N$ and $t>0$, and replaced it with the requirement that the functions $\varphi_{n+1}-\varphi_n$ are non-increasing on $[0,+\infty)$ for $n \in \mathbb N$. The sole reason behind our decision is that without this assumption the proof of Helly's selection theorem for the Schramm variation (presented in~\cites{ABM,schramm}) may be not correct. The proof is based on a very technical result concerning some pointwise estimates of the so-called variation function of a given function $x$ such that $\var_{\Phi} x<+\infty$ (see~\cite{schramm}*{Lemma~2.5}). A starting point to obtain those estimates (technical issues aside) is to choose, for a given $\varepsilon>0$, a finite collection $I_1,\ldots,I_k$ of closed and non-overlapping subintervals of $[0,1]$ such that $\abs{x(I_1)}\geq \abs{x(I_2)}\geq \ldots \geq \abs{x(I_n)}$ and
\begin{equation}\label{eq:estimate_var}
 \var_{\Phi} x \leq \varepsilon + \sum_{n=1}^k \varphi_n\bigl(\abs{x(I_n)}\bigr).
\end{equation}
However, in general, it is not possible to control the monotonicity of the sequence $\bigl(\abs{x(I_n)}\bigr)_{n=1}^k$, guaranteeing at the same time that the estimate~\eqref{eq:estimate_var} holds, unless we know that the differences $\varphi_{n+1}-\varphi_n$ are non-increasing. To see this, let us consider the sequence $(\varphi_n)_{n \in \mathbb N}$ of Young functions given by  
\[
 \varphi_1(t)=\begin{cases}
               t, & \text{if $t \in [0,1]$,}\\
							2t-1, & \text{if $t \in (1,+\infty)$,}
							 \end{cases} \quad \text{and} \quad  \varphi_n(t)=\begin{cases}
               t^2, & \text{if $t \in [0,1]$,}\\
							2t-1, & \text{if $t \in (1,+\infty)$,}
							 \end{cases}
\]
for $n\geq 2$. Moreover, let $x \colon [0,1] \to \mathbb R$ be given by 
\[
 x(t)=\begin{cases}
        \frac{3}{4}, & \text{if $t=0$,}\\
				0, & \text{if $t \in (0,1)$,}\\
				\frac{1}{2}, & \text{if $t=1$}.
				\end{cases}
\]
If $I_1,\ldots,I_k$ is a finite collection of closed, non-degenerate and non-overlapping subintervals of $[0,1]$, then in the sum $\sum_{n=1}^k \varphi_n\bigl(\abs{x(I_n)}\bigr)$ at most two summands are non-zero, and those summands correspond to those intervals which contain the points $t=0$ and $t=1$. In other words, in order to maximize $\sum_{n=1}^k \varphi_n\bigl(\abs{x(I_n)}\bigr)$ we may assume that $0,1 \in I_1\cup I_2$. If $0,1$ belong to the same interval, then the collection $I_1,\ldots,I_k$ reduces to the interval $I_1:=[0,1]$ and
\[
 \sum_{n=1}^k \varphi_n\bigl(\abs{x(I_n)}\bigr)=\abs{x(I_1)}=\frac{1}{4}.
\]
On the other hand, if $0\in I_1$ and $1 \in I_2$, then 
\[
\sum_{n=1}^k \varphi_n\bigl(\abs{x(I_n)}\bigr) = \abs{x(I_1)} + \abs{x(I_2)}^2 = \frac{3}{4}+\frac{1}{4}=1. 
\]
Finally, if $1\in I_1$ and $0 \in I_2$, then 
\[
\sum_{n=1}^k \varphi_n\bigl(\abs{x(I_n)}\bigr) = \abs{x(I_1)} + \abs{x(I_2)}^2 = \frac{1}{2}+\frac{9}{16}=\frac{17}{16}. 
\]
Thus, $\var_\Phi x = \frac{17}{16}$ (cf. Proposition~\ref{prop:variation_various_definitions}). This also shows that for any $\varepsilon \in (0,\frac{1}{16})$ the only viable collection of closed and non-overlapping subintervals of $[0,1]$ so that~\eqref{eq:estimate_var} holds is: $J_1:=[a,1]$ and $J_2:=[0,b]$ for any $0<b\leq a <1$. But then $\abs{x(J_1)}=\frac{1}{2}<\frac{3}{4}=\abs{x(J_2)}$.

The above example is a manifestation of a general rule concerning rearrangement inequalities for finite sequences. In~\cite{vince} Vince proved the following result. \emph{Let $g_1,\ldots,g_k$ be real-valued functions defined on an interval $J$. Then,
\[
 \sum_{n=1}^k g_n(b_{k-n+1})\leq \sum_{n=1}^k g_n(b_{\pi(n)}) \leq \sum_{n=1}^k g_n(b_n)
\]
for any sequence $b_1\leq b_2\leq \ldots\leq b_k$ in $J$ and any permutation $\pi$ of the set $\{1,\ldots,k\}$ if and only if the functions $g_{n+1}-g_n$ are non-decreasing on $J$ for $n=1,\ldots,k-1$.}

Fortunately, the assumption that for a given sequence of Young functions $(\varphi_n)_{n \in \mathbb N}$ the increments $\varphi_{n+1}-\varphi_n$ are non-increasing on $[0,+\infty)$ is quite natural, as it is satisfied not only in all the classical cases described in Remark~\ref{rem:variations}, but also in the case of the so-called $\phi\Lambda$-variation and $\Lambda$-variation of order $p$ studied in~\cite{SW} and~\cite{shiba}, respectively. 
\end{remark}

We end this subsection with a compactness criterion in $\Phi BV[0,1]$, which is a consequence of the previous results, the abstract results proved in Section~\ref{sec:abstract_compactness_criterion} and an easy fact that $\Phi BV[0,1]$ is a Banach space in the norm $\norm{\cdot}_{\Phi}$ (cf.~\cite{schramm}*{Theorem~2.3} and see~\cite{ABM}*{Proposition~2.44 (e)}). (Observe also that if a family of semi-norms satisfies condition~\ref{iii}, then it also satisfies condition~\ref{iv}.)

\begin{theorem}\label{thm:compactnes_PBV}
Let $(\varphi_n)_{n \in \mathbb N}$ be a fixed sequence of Young functions such that the functions $\varphi_{n+1}-\varphi_n$ are non-increasing on $[0,+\infty)$ for $n \in \mathbb N$ and $\sum_{n=1}^\infty \varphi_n(t)=+\infty$ for each $t >0$. For any bounded and non-empty subset $A$ of $(\Phi BV[0,1], \norm{\cdot}_\Phi)$ the following conditions are equivalent\textup:
\begin{enumerate}[label=\textup{(\roman*)}]
 \item $A$ is relatively compact in $(\Phi BV[0,1], \norm{\cdot}_\Phi)$,
 \item for every $\varepsilon>0$ and every $y \in \Phi BV[0,1]$ there exists a finite family $\mathcal J$ of closed subintervals of $[0,1]$ such that $\abs{x-y}_\Phi \leq \varepsilon + \abs{x-y}_\mathcal J$ for every $x \in A$,
 \item for every $\varepsilon>0$ there exists a finite family $\mathcal J$ of closed subintervals of $[0,1]$ such that $\abs{x-y}_\Phi \leq \varepsilon + \abs{x-y}_\mathcal J$ for all $x,y \in A$\textup;
\end{enumerate}
here the semi-norms $\abs{\cdot}_\Phi$ and $\abs{\cdot}_{\mathcal J}$ are defined by the formula~\eqref{eq:semi-norm}.
\end{theorem}

\begin{remark}
Note that the results of this subsection extend the results obtained by Bugajewski and Gulgowski in~\cites{BG20, G21} concerning (relative) compactness in the spaces of functions of bounded Jordan, Young and Waterman variations (cf. Remark~\ref{rem:variations}).
\end{remark}

We end this paper with yet another open problem.

\begin{question}
The reason behind strengthening the monotonicity assumption on the sequence $(\varphi_n)_{n \in \mathbb N}$ of Young functions in all the compactness results in the space $\Phi BV[0,1]$ was we explained in detail in Remark~\ref{rem:explanation_why_add_assumption}. Long story short, without this stronger assumption it is not clear whether Helly's selection theorem for functions of bounded Schramm variation (that is, Theorem~2.8 in~\cite{schramm}) holds. It is, therefore, interesting whether it is possible to prove Helly's selection theorem (and consequently all the compactness results established above) under the ``original'' assumption that the sequence $(\varphi_n)_{n \in \mathbb N}$ is such that $\varphi_{k+1}(t)\leq \varphi_k(t)$ and $\sum_{n=1}^\infty \varphi_n(t)=+\infty$ for each $t >0$ and $k \in \mathbb N$.  
\end{question}

\begin{ackn}
We would like to thank Professor Dariusz Bugajewski for his questions and comments concerning the properties of equi-normed sets and for the opportunity to present our research at the Nonlinear Analysis Seminar held at the Faculty of Mathematics and Computer Science of Adam Mickiewicz University in Pozna\'n.

We are also very grateful to the anonymous Referee for his/her valuable comments that helped to improve the paper.
\end{ackn}

\end{document}